\documentclass[11pt,reqno]{amsart}
\usepackage{amssymb}
\usepackage{bm}
\usepackage{xcolor}
\usepackage{amsmath}
\usepackage{enumitem}
\usepackage{mathrsfs}
\usepackage{hyperref}
\usepackage{titlesec}

\titleformat{\section}  
{\fontsize{14}{16}\bfseries} 
{\thesection.} 
{1em} 
{} 
[] 
\titlespacing*{\section}{0pt}{1.5em}{0.8em}

\titleformat{\subsection}  
{\fontsize{13}{15}\bfseries} 
{\thesubsection.} 
{1em} 
{} 
[] 
\titlespacing*{\subsection}{0pt}{1.2em}{0.6em}

\newtheorem{theorem}{Theorem}[section]
\newtheorem{lemma}[theorem]{Lemma}

\newtheorem{fact}[theorem]{Fact}
\newtheorem{question}[theorem]{Question}

\newtheorem{proposition}[theorem]{Proposition}
\newtheorem{corollary}[theorem]{Corollary}
\newtheorem{claim}[theorem]{Claim}

\theoremstyle{definition}
\newtheorem{definition}[theorem]{Definition}
\newtheorem{notation}[theorem]{Notation}
\newtheorem{remark}[theorem]{Remark}

\theoremstyle{definition}

\let \iff = \leftrightarrow
\let \sub = \subseteq

\DeclareMathOperator{\dom}{dom}

\DeclareMathOperator{\range}{range}
\DeclareMathOperator{\field}{field}

\DeclareMathOperator{\pred}{pred}

\title{Long Strong Chains of Subsets of $\omega_1$}

\author[David Asper\'o]{David Asper\'o}

\address{David Asper\'o, School of Engineering, Mathematics, and Physics, University of East Anglia, Norwich NR4 7TJ, UK}

\email{d.aspero@uea.ac.uk}

\author[Curial Gallart]{Curial Gallart}

\address{Curial Gallart, Institute of Mathematics, University of Vienna, Kolingasse 14-16, 1090 Vienna, Austria}

\email{curial.gallart.rodriguez@univie.ac.at}

\date{}

\begin{document}
	
	
	
	\thanks{The second author was partially supported by EPSRC project reference 2440142 and NCN-FWF Weave project PIN 1355423. This work was part of the second author's PhD thesis, written under the supervision of the first author at the University of East Anglia.}

	\begin{abstract}
		We force the existence of a chain of length $\omega_3$ in $[\omega_1]^{\omega_1}$ increasing modulo finite. The construction involves symmetric systems of models of two types as side conditions, introduced by the second author. This improves previous results of Koszmider and Veli\v{c}kovi\'{c}-Venturi.
	\end{abstract}
	
	\maketitle
	\pagestyle{myheadings}\markright{Long Strong Chains of Subsets of $\omega_1$}

	\section{Introduction}
	
	The study of the spaces $^\omega2$, $^\omega\omega$ and $[\omega]^\omega$ has been a central topic in the development of set theory since its beginnings. Their importance comes from the fact that when equipped with the right topology, they are homeomorphic to the real numbers minus some countable subset. Therefore, in many regards, they are equivalent to $\mathbb{R}$.\footnote{It is even common practice in set theory to refer to the elements of these spaces as \emph{reals}.} Moreover, for many different reasons, these spaces are more suitable than $\mathbb{R}$ from the point of view of forcing theory. Of great importance is the area of \emph{cardinal characteristics of the continuum}, which deals with the possible cardinalities of certain subsets of the the real line when the continuum is assumed to be larger than $\aleph_1$ (see \cite{Blass2010CCCC} for a survey on the topic). In the past decades there has been an increasing interest in the spaces of \emph{higher} (or \emph{generalized}) \emph{reals} $^\lambda2$, $^\lambda\lambda$ and $[\lambda]^\lambda$, where $\lambda$ is an uncountable cardinal, and the \emph{higher cardinal characteristics} associated to these spaces (see, for example \cite{QuestionsongeneralisedBairespaces}, \cite{CummingsShealh1995Cardinalinvariantsabovethecontinuum}, or \cite{BrendleBrookeTaylorFriedmanMontoya2018Cichonsdiagramforuncountablecardinals}). What makes this topic so interesting is that, in many cases, it turns out that the classical theory does not generalize to the uncountable case as straightforwardly as one would expect. In fact, in some cases many results are only known to work at the level of some large cardinal. Hence, this area has motivated the development of many deep and powerful ideas in forcing theory and, moreover, it has given completely new characterizations of known large cardinal notions. 
	
	In the area of higher cardinal characteristics one usually considers spaces such as $[\lambda]^{\lambda}$ quotiented by the ideal of subsets of $\lambda$ of size $<\lambda$. This fact is crucial when generalizing the classical theory of cardinal characteristics to the uncountable. In fact, the space $[\lambda]^\lambda$ quotiented by the ideal of subsets of $\lambda$ of size $<\mu$, for some cardinal $\mu<\lambda$, is much harder to deal with. Indeed, the known techniques coming from the area of cardinal characteristics do not seem to work in this context, so our understanding of this space is much poorer.
	In this paper we will focus on the space $[\omega_1]^{\omega_1}$ quotiented by the ideal of finite subsets of $\omega_1$. Let us first give a brief overview of the main results in the study of the spaces $[\omega_1]^{\omega_1}$ and $^{\omega_1}\omega_1$ in this
	 context.
	
	
	\begin{definition}\label{def-almost-disjoint-subsets}
		If $\delta$ is a limit ordinal, a \emph{strong $\delta$-almost disjoint family of subsets of $\omega_1$} is a sequence $\langle A_\alpha:\alpha<\delta\rangle$ of subsets of $\omega_1$ such that for all $\alpha<\beta<\delta$, 
		\begin{enumerate}[label=(\arabic*)]
			\item $|A_\alpha|=\aleph_1$, and
			
			\item $|A_\alpha\cap A_\beta|<\aleph_0$.
		\end{enumerate}
	\end{definition}
	
	Baumgartner showed in \cite{Baumgartner1976Almostdisjointsetsthedensesetproblemandthepartitioncalculus} that one
	can consistently have arbitrarily long strong almost disjoint families of subsets of $\omega_1$.
	
	\begin{theorem}[Baumgartner]\label{thm-Baumgartner}
		Assume $\mathrm{GCH}$ and let $\delta$ be a limit ordinal. There exists a cardinal-preserving forcing notion $\mathbb{P}$ that forces the existence of a strong $\delta$-almost disjoint family of subsets of $\omega_1$.
	\end{theorem}
	
	
		
	
	
	Baumgartner's result motivated the introduction of stronger properties for families living in the spaces $[\omega_1]^{\omega_1}$ and $^{\omega_1}\omega_1$.
	
	\begin{definition}\label{def-chain-subsets}
		A \emph{strong $\delta$-chain of subsets of $\omega_1$} is a sequence $\langle X_\alpha:\alpha<\delta\rangle$ of subsets of $\omega_1$ such that for all $\alpha<\beta<\delta$, 
		\begin{enumerate}[label=(\arabic*)]
			\item $|X_\beta\setminus X_\alpha|=\aleph_1$, and
			
			\item $|X_\alpha\setminus X_\beta|<\aleph_0$.
		\end{enumerate}
	\end{definition}
	
	In the same paper, Baumgartner showed that the existence of a strong chain of subsets of $\omega_1$ implies the existence of a strong almost disjoint family of subsets of $\omega_1$. More precisely, if $\langle X_\alpha:\alpha<\delta\rangle$ is a strong $\delta$-chain of subsets of $\omega_1$ and we define $A_\alpha:=X_{\alpha+1}\setminus X_\alpha$, for every $\alpha<\delta$, then $\langle A_\alpha:\alpha<\delta\rangle$ is a strong $\delta$-almost disjoint family of subsets of $\omega_1$.
	
	\begin{definition}\label{def-chain-functions}
		If $\delta$ is a limit ordinal, a \emph{strong $\delta$-chain of functions from $\omega_1$ to $\omega_1$} is a sequence $\langle f_\alpha:\alpha<\delta\rangle$ of functions such that for all $\alpha<\beta<\delta$, 
		\begin{enumerate}[label=(\arabic*)]
			\item $f_\alpha\in {^{\omega_1}\omega_1}$, and
			
			\item $|\{\xi\in\omega_1:f_\alpha(\xi)\geq f_\beta(\xi)\}|<\aleph_0$.
		\end{enumerate}
	\end{definition}
	
	The existence of a strong chain of functions from $\omega_1$ to $\omega_1$ implies the existence of a strong chain of subsets of $\omega_1$, by identifying each subset of $\omega_1$ with its characteristic function. 
	
	Strong chains of length $\omega_1$ exist in $\mathrm{ZFC}$. However, the existence of longer strong chains of functions implies the failure of $\mathrm{CH}$, and hence their existence is independent of $\mathrm{ZFC}$. Hajnal and Szentmikl\'{o}ssy asked in \cite{Finiteandinfinitecombinatoricsinsetsandlogic1993} (s.\ page 435), the following two natural questions about the existence of long strong chains.
	
	\begin{question}
		Is it consistent that there exists a strong $\omega_2$-chain of subsets of $\omega_1$?
	\end{question}
	
	And the harder version of the question.
	
	\begin{question}
		Is it consistent that there exists a strong $\omega_2$-chain of functions from $\omega_1$ to $\omega_1$? 
	\end{question}
	
	Both questions were answered affirmatively by Koszmider.
	
	\begin{theorem}[Koszmider, \cite{Koszmider1998ontheexistenceofstrongchainsin}]
		Assuming $\square_{\omega_1}$, there exists a c.c.c. forcing poset that adds a strong $\omega_2$-chain of subsets of $\omega_1$.
	\end{theorem}
	
	Moreover, in the same paper he proves that
	 Chang's Conjecture implies that there is no strong $\omega_2$-chain of subsets of $\omega_1$ (and hence no strong $\omega_2$-chain of functions from $\omega_1$ to $\omega_1$ either). Furthermore, since Chang's Conjecture is preserved by c.c.c.\ forcings, the following theorem shows that being a strong chain of subsets of $\omega_1$ is a strictly stronger notion than that of being a strong almost disjoint family.
	
	\begin{theorem}[Koszmider, \cite{Koszmider1998ontheexistenceofstrongchainsin}]
		It is consistent, modulo large cardinals, that Chang's Conjecture holds together with the existence of a strong almost disjoint family of subsets of $\omega_1$ and the nonexistence of strong chains of subsets of $\omega_1$. 
	\end{theorem}
	
	Koszmider answered the second question affirmatively in \cite{Koszmider2000onstrongchainsofuncoutablefunctions}. However, his proof required completely new ideas; in fact, in the same paper he proved that, assuming $\mathrm{CH}$, there is no c.c.c.\ poset forcing the existence of a strong $\omega_2$-chain of functions from $\omega_1$ to $\omega_1$. In particular, this implies that one can consistently have long strong chains of subsets of $\omega_1$, while having no strong chains of functions from $\omega_1$ to $\omega_1$.
	
	\begin{theorem}[Koszmider, \cite{Koszmider2000onstrongchainsofuncoutablefunctions}]
		It is consistent that there is a strong $\omega_2$-chain of subsets of $\omega_1$ but there is no strong $\omega_2$-chain of functions from $\omega_1$ to $\omega_1$.
	\end{theorem}
	
	In order to answer the second question, Koszmider introduced a cardinal-preserving forcing notion with side conditions organized along morasses. These side conditions, which consist of finitely many elements of a simplified $(\omega_1,1)$-morass (\cite{Velleman1984simplifiedmorasses}), resemble the \emph{symmetric systems} of \cite{Aspero:Mota2015a} or \cite{KuzeljevicTodorcevic2017Forcingwithmatrices}, and ensure that the forcing poset is proper (for a stationary set of countable elementary submodels) and has the $\aleph_2$-chain condition, provided that $\mathrm{CH}$ holds in the ground model.  
	
	
	\begin{theorem}[Koszmider, \cite{Koszmider2000onstrongchainsofuncoutablefunctions}]\label{Koszmider}
		Assuming $\mathrm{CH}$, there exists a cardinal-preserving forcing poset that adds a strong $\omega_2$-chain of functions from $\omega_1$ to $\omega_1$.
	\end{theorem}
	
	We can naturally generalize the notions of strong chains of sets and functions by considering sequences of elements of the spaces $[\lambda]^{\lambda}$ and $^\lambda\lambda$, respectively, increasing modulo the ideal $[\lambda]^{\mu}$, for some $\mu<\lambda$. However, Shelah and Inamdar proved some impossibility results about the existence of strong chains in higher spaces. First, Shelah \cite{Shelah2010Onlongincreasingchainsmoduloflatideals} showed that for every infinite cardinal $\lambda$, there cannot be sequences $\langle f_\alpha:\alpha<\lambda^{+++}\rangle$ of functions from $^{\lambda^{++}}\lambda^{++}$ increasing modulo the ideal $[\lambda^{++}]^{<\lambda}$. Later, Inamdar \cite{Inamdar2023Onstrongchainsofsetsandfunctions} improved Shelah's result by showing that there cannot be sequences $\langle X_\alpha:\alpha<\lambda^{+++}\rangle$ of subsets of $\lambda^{++}$ increasing modulo the ideal $[\lambda^{++}]^{<\lambda}$. In particular, there cannot be sequences of subsets of $\omega_2$ of length $\omega_3$ increasing modulo finite.
	
	Despite the limitations imposed by the results of Shelah and Inamdar, the following very general question about the spaces $[\omega_1]^{\omega_1}$ and $^{\omega_1}\omega_1$ remains open.
	
	\begin{question}
		What are the possible lengths of strong chains of functions from $\omega_1$ to $\omega_1$ and subsets of $\omega_1$?
	\end{question}
	
	In the wake of Neeman's introduction of side conditions consisting of finite $\in$-chains of models of two types (\cite{Neeman2014Forcingwithsequencesofmodelsoftwotypes}, \cite{Neeman2017Twoapplicationsoffinitesideconditionsatomega2}), Veli\v{c}kovi\'{c} and Venturi \cite{Velickovic:Venturi2011properforcingremastered} obtained a much simpler proof of Koszmider's result. Their forcing poset incorporates Neeman's side conditions to ensure the preservation of $\omega_1$ and $\omega_2$. However, the preservation of larger cardinals is, in general, not guaranteed by 
	 this type of side conditions.\footnote{Although it is possible to ensure the $\aleph_3$-chain condition, in some cases, by considering large enough transitive-type models (\cite{Neeman2017Twoapplicationsoffinitesideconditionsatomega2}).} Therefore, the problem of finding strong chains in $[\omega_1]^{\omega_1}$ and $^{\omega_1}\omega_1$ of length $\omega_3$ seems to be out of reach of Neeman's method. Therefore, in addition to being a very important question in combinatorial set theory, it has also been regarded as a test question for finding the right notion of side conditions of models of three types.
	
	In this paper, we will show that side conditions of models of two types are enough to force a strong $\omega_3$-chain of subsets of $\omega_1$. In particular, we will define a forcing poset that incorporates symmetric systems of models of two types as side conditions. These side conditions, which were introduced by the second author in \cite{Gallart2026forcingsymmetricsystemsmodels}, will ensure that the forcing is proper and $\aleph_1$-proper, and that it has the $\aleph_3$-Knaster condition, under the right assumptions. The following theorem is the main result of the paper.
	
	\begin{theorem}\label{mainthm}
		If $\mathrm{GCH}$ holds, then there is a cardinal-preserving forcing notion $\mathbb{P}$ forcing the existence of a strong $\omega_3$-chain of subsets of $\omega_1$.
	\end{theorem}
	
	The proof of the above theorem is built on some of the ideas of Koszmider (\cite{Koszmider2000onstrongchainsofuncoutablefunctions}) and Veli\v{c}kovi\'{c}-Venturi (\cite{Velickovic:Venturi2011properforcingremastered}), and we believe that a very mild modification of the forcing witnessing the theorem should allow us to get the same result for strong chains of functions from $\omega_1$ to $\omega_1$.
	
	The paper is organized as follows. In Section \ref{section-prelim}, we will establish the notation and provide all the necessary background in (strongly) proper forcing and elementary submodels. In section \ref{section-pure}, we will introduce symmetric systems of models of two types and give an overview of the main amalgamation lemmas from \cite{Gallart2026forcingsymmetricsystemsmodels}. We will not give a full account of the proofs of these lemmas, but we will give a detailed sketch of the arguments. In Section \ref{section-strong-chains}, we will define the forcing poset witnessing Theorem \ref{mainthm}. We will start by proving some density lemmas and basic properties of the forcing, and then we will prove the cardinal preservation lemmas, which are the most involved results of the section. In particular, we will show that, starting from a model of $\mathrm{GCH}$, the forcing is proper, $\aleph_1$-proper and has the $\aleph_3$-Knaster condition.

	\section{Preliminaries}\label{section-prelim}
	
	Our notation will be standard and follows \cite{Kunen2011Settheory} and \cite{Jech2003Settheory}. We refer the reader to these two sources for any undefined notions. Unless otherwise specified, lowercase Greek letters $\alpha,\beta,\gamma,\delta,\varepsilon,\xi,\eta$ will be used to denote ordinals, while $\kappa,\lambda,\mu,\nu,\theta$ will be used to denote infinite cardinals. We will denote by $\mathrm{OR}$ the class of all ordinals. Let $X$ be any set. If $\mu$ is a cardinal, we will denote by $[X]^\mu$ the set of all subsets of $X$ of size $\mu$. The sets $[X]^{<\mu}$ and $[X]^{\leq\mu}$ are defined in the obvious way. If $f$ is a function and $X\subseteq\dom(f)$, then $f"(X)$ denotes the set $\{f(x):x\in X\}$. 
	
	Recall that for every set $X$, a subset $U\subseteq\mathcal{P}(X)$ is \emph{unbounded in $\mathcal{P}(X)$} (or \emph{unbounded in $X$}) if it is unbounded with respect to inclusion. A subset $C\subseteq\mathcal{P}(X)$ is \emph{club in $\mathcal{P}(X)$} (or \emph{club in $X$}) if it is closed and unbounded with respect to inclusion. A subset $S\subseteq\mathcal{P}(X)$ is \emph{stationary in $\mathcal{P}(X)$} (or \emph{stationary in $X$}) if it has non-empty intersection with all clubs in $\mathcal{P}(X)$. More generally, $C\subseteq\mathcal{P}(X)$ is \emph{club in $\mathcal{P}(X)$} if there exists a function $f:[X]^{<\omega}\to X$ such that for every $x\in \mathcal{P}(X)$, $x\in C$ if and only if $f"([x]^{<\omega})\subseteq x$. Therefore, a subset $S\subseteq\mathcal{P}(X)$ is \emph{stationary in $\mathcal{P}(X)$} if for every function $f:[X]^{<\omega}\to X$, there is $x\in S$ which is closed under $f$, i.e., $f"([x]^{<\omega})\subseteq x$. The two notions of club coincide when replacing $\mathcal{P}(X)$ with $[X]^{\omega}$.
	
	\begin{definition}
		Let $\kappa$ be an uncountable cardinal. A poset $\mathbb{P}$ has the \emph{$\kappa$-Knaster condition} if for every $A\subseteq\mathbb{P}$ of size $\kappa$ there is a subset $B\subseteq A$ of the same size consisting of pairwise compatible conditions.
	\end{definition}
	
	\begin{remark}
		Every poset $\mathbb{P}$ with the $\kappa$-Knaster condition has the $\kappa$-chain condition (denoted $\kappa$-c.c.).
	\end{remark}
	
	\begin{lemma}
		If $\kappa$ is an uncountable cardinal and $\mathbb{P}$ is a forcing notion with the $\kappa$-c.c., then every cardinal $\lambda\geq\kappa$ is preserved after forcing with $\mathbb{P}$.
	\end{lemma}
	
	We review some well-known general facts about proper and strongly proper forcing. 
	

	\begin{definition}
		Let $Q$ be an elementary submodel of a big enough $H(\theta)$, let $\mathbb{P}$ be a forcing poset and let $p$ be a condition in $\mathbb{P}$.
		\begin{enumerate}
			\item $p$ is \emph{$(Q,\mathbb{P})$-generic} if for every dense subset $D\subseteq\mathbb{P}$ such that $D\in Q$, the set $D\cap Q$ is predense below $p$.
			\item $p$ is \emph{strongly $(Q,\mathbb{P})$-generic} if for every dense subset $D\subseteq\mathbb{P}\cap Q$, the set $D$ is predense below $p$.
		\end{enumerate}
	\end{definition}
	
	\begin{remark}
		Note that if a condition $p$ in a poset $\mathbb{P}$ is strongly $(Q,\mathbb{P})$-generic, where $Q$ is an elementary submodel of some $H(\theta)$ such that $\mathbb{P}\in Q$, then $p$ is $(Q,\mathbb{P})$-generic.
	\end{remark}
	
	
	\begin{remark}\label{remark-str-gen}
		Let $\kappa$ be an infinite cardinal and let $\mathbb{P}\subseteq H(\kappa)$ be a forcing notion. Let $p\in\mathbb{P}$ and $Q\preceq H(\kappa)$ such that $p$ is strongly $(Q,\mathbb{P})$-generic. If $\theta>\kappa$ is a large enough cardinal and $Q^*\preceq H(\theta)$ is such that $Q^*\cap H(\kappa)=Q$, then $p$ is also strongly $(Q^*,\mathbb{P})$-generic.  
	\end{remark}
	
	\begin{definition}
		Let $\theta$ be an infinite cardinal, let $\mathcal{K}$ be a collection of elementary submodels of $H(\theta)$, and let $\mathbb{P}$ be a forcing notion.
		\begin{enumerate}
			\item $\mathbb{P}$ is \emph{$\mathcal{K}$-proper} if for every $Q\in\mathcal{K}$ such that $\mathbb{P}\in Q$ and every $p\in\mathbb{P}\cap Q$, there is a $(Q,\mathbb{P})$-generic condition $q\in\mathbb{P}$ extending $p$.
			
			\item $\mathbb{P}$ is \emph{strongly $\mathcal{K}$-proper} if for every $Q\in\mathcal{K}$ such that $\mathbb{P}\in Q$ and every $p\in\mathbb{P}\cap Q$, there is a strongly $(Q,\mathbb{P})$-generic condition $q\in\mathbb{P}$ extending $p$.
		\end{enumerate}
		
		If $\mu$ is an infinite cardinal and $\mathcal{K}$ is the collection of all elementary submodels $Q\preceq H(\theta)$ such that $|Q|=\mu$ and $^{<\mu}Q\subseteq Q$, then we call $\mathbb{P}$ \emph{$\mu$-proper} or \emph{strongly $\mu$-proper}, respectively. Moreover, if $\mu=\aleph_0$, then we call $\mathbb{P}$ \emph{proper} or \emph{strongly proper}, respectively.
	\end{definition}
	
	\begin{lemma}[Claim 3.4 in \cite{Neeman2014Forcingwithsequencesofmodelsoftwotypes}]\label{preservacio-proper}
		Let $\theta$ be an infinite cardinal and let $\mathcal{K}^*$ be a collection of elementary submodels of $H(\theta)$. Suppose that $\mathbb{P}\in H(\theta)$ is a $\mathcal{K}^*$-proper forcing notion. Let $\lambda$ be a cardinal and assume that $$\{Q^*\in\mathcal{K}^*:\alpha\subseteq Q^*,|Q^*|<\lambda\}$$ is stationary in $H(\theta)$ for each $\alpha<\lambda$. Then the forcing $\mathbb{P}$ preserves $\lambda$. 
	\end{lemma}
	
	\begin{corollary}[Claim 3.5 in \cite{Neeman2014Forcingwithsequencesofmodelsoftwotypes}]\label{preservacio-proper2}
		Let $\kappa$ be an infinite cardinal and let $\mathcal{K}$ be a collection of elementary submodels of $H(\kappa)$. Suppose that $\mathbb{P}\subseteq H(\kappa)$ is a strongly $\mathcal{K}$-proper forcing notion. Let $\lambda$ be a cardinal and assume that $$\{Q\in\mathcal{K}:\alpha\subseteq Q,|Q|<\lambda\}$$ is stationary in $H(\kappa)$ for each $\alpha<\lambda$. Then the forcing $\mathbb{P}$ preserves $\lambda$. 
	\end{corollary}

	\subsection{Elementary submodels}
	
	Given a model $Q$, we will denote $Q\cap\omega_1$ by $\delta_Q$ and $\sup(Q\cap\omega_2)$ by $\varepsilon_Q$, and we will call $\delta_Q$ the \textit{$\omega_1$-height of $Q$} and $\varepsilon_Q$ the \textit{$\omega_2$-height of $Q$}. 
	
	Given two $\in$-isomorphic models of the Axiom of Extensionality $Q_0$ and $Q_1$, we write $\Psi_{Q_0,Q_1}$ to denote the unique isomorphism $\Psi$ between the structures $(Q_0;\in)$ and $(Q_1;\in)$.
	
	For the rest of the paper we will fix a cardinal $\kappa>\omega_2$ and a predicate $T\subseteq H(\kappa)$. We will usually refer to the structure $(H(\kappa);\in,T)$ simply by $H(\kappa)$. Let $\mathcal{S}$ be the collection of countable $M\preceq(H(\kappa); \in, T)$. We will tend to use capital letters $M,N$ to refer to models in $\mathcal{S}$, which we will call \emph{countable elementary} or \textit{small models}. It is a well-known result that $\mathcal{S}$ contains a club subset of $[H(\kappa)]^\omega$.
	
	\begin{definition}
		We will call a collection $\mathcal{L}$ of $\aleph_1$-sized elementary submodels $X\preceq (H(\kappa); \in, T)$ \emph{appropriate for $\mathcal{S}$}, if for every $X\in\mathcal{L}$ and every $M\in\mathcal{S}$ such that $X\in M$, then $X\cap M\in X\cap\mathcal{S}$.
	\end{definition}
	
	We will use $\mathcal{L}$ to denote arbitrary collections of models of size $\aleph_1$ appropriate for $\mathcal{S}$, and the capital letters $X,Y,Z$ to refer to the models in $\mathcal{L}$, which we will call \emph{uncountable} or \textit{large models}. 
	
	
	
	
	\begin{fact}
		Assuming $\mathrm{CH}$, the collection of all elementary submodels $X\preceq H(\kappa)$ such that $|X|=\aleph_1$ and $^\omega X\subseteq X$ is appropriate for $\mathcal{S}$ and stationary in $H(\kappa)$. 
	\end{fact}
	
	
	To refer to models of arbitrary size, we will use $Q$, as well as other capital letters further down the alphabet. If $Q$ is an elementary submodel of $H(\kappa)$, we will usually refer to the structure $(Q;\in,T\cap Q)$ by $(Q;\in,T)$. Moreover, we might indistinctly use $Q$ to refer to the structure $(Q;\in,T)$ or its universe. It will be clear from the context to which one we are referring to.
	
	The following are some basic facts about elementary submodels, which can be found in any standard reference for set-theoretic notions such as \cite{Kunen2011Settheory} and \cite{Jech2003Settheory}. We include them without proof, and will use them throughout the paper, sometimes without mention. Additional information on elementary submodels and their applications in set theory can be found in \cite{Dow1998Anintroductiontoapplicationsofelementarysubmodels} and \cite{JustWeese1997DiscoveringmodernsetthoryII}.
	
	\begin{theorem}[Tarski-Vaught test]
		Let $M$ be a model and let $A\subseteq M$. Then, $A$ is the domain of an elementary submodel $N\preceq M$ iff for every formula $\varphi(y,\overline{x})$ and every tuple $\overline{a}$ of $A$ such that $M\models\exists y\varphi(y,\overline{a})$, there is $b\in A$ such that $M\models\varphi(b,\overline{a})$.
	\end{theorem}
	
	\begin{proposition}
		If $Q_0,Q_1\preceq H(\kappa)$ are such that $Q_0\subseteq Q_1$, then $Q_0\preceq Q_1$.
	\end{proposition}
	
	\begin{proposition}
		Let $Q\preceq H(\kappa)$, and let $\mu<\kappa$ be a cardinal such that $\mu\subseteq Q$. For every $A\in Q$, if $|A|=\mu$,
		 then $A\subseteq Q$.
	\end{proposition}
	
	\begin{proposition}
		Let $Q\preceq H(\kappa)$. If $A$ is definable over $H(\kappa)$ with parameters in $Q$, then $A\in Q$.
	\end{proposition}
	
	\begin{proposition}
		Let $\theta>\kappa$, $Q^*\preceq H(\theta)$ and $\kappa\in Q^*$. Then $Q^*\cap H(\kappa)$ is an elementary submodel of $H(\kappa)$.
	\end{proposition}
	
	\begin{proposition}
		Let $Q$ be an elementary submodel of $H(\kappa)$ such that $|Q|=\mu<\mu^+<\kappa$. Then $Q\cap\mu^+\in\mu^+$ is a limit ordinal.
	\end{proposition}
	

\begin{proposition}
		Let $Q_0,Q_1\preceq H(\kappa)$ such that $|Q_0|=|Q_1|=\mu<\mu^+<\kappa$ and $\mu\subseteq Q_0\cap Q_1$ and let $\Psi$ be an isomorphism between $(Q_0;\in,T)$ and $(Q_1;\in,T)$. Suppose $T$ codes a sequence $(e_\alpha\,:\,\alpha\in \mu^+\setminus\{0\})$ such that each $e_\alpha$ is a surjection from $\mu$ onto $\alpha$. Then $\Psi$ is the identity on $Q_0\cap\mu^+$. In particular, $Q_0\cap\mu^+=Q_1\cap\mu^+$.
	\end{proposition}
	
	\begin{proposition}\label{prop5}
		Let $Q_0,Q_1$ and $P$ be elementary submodels of $H(\kappa)$. Suppose that $P\in Q_0$ and $P\subseteq Q_0$, and that $\Psi:(Q_0;\in,T)\to(Q_1;\in,T)$ is an isomorphism. Then $\Psi(P)$ is an elementary submodel of $(H(\kappa);\in,T)$. 
	\end{proposition}
	\begin{proof}
		It is straightforward to see that $\Psi\upharpoonright P$ is an isomorphism between $(P;\in,T)$ and $(\Psi(P);\in,T)$. Assume now that $\varphi(y,\bar{x})$ is a first-order formula in the language of set theory and let $\Psi(\bar{a})$ be a tuple of elements of $\Psi(P)$ such that $H(\kappa)$ satisfies the formula $\exists y\varphi\big(y,\Psi(\bar{a})\big)$. Since $Q_1\preceq H(\kappa)$ and $\Psi(P)\subseteq Q_1$, the model $Q_1$ also satisfies the formula $\exists y\varphi\big(y,\Psi(\bar{a})\big)$, and since $\Psi$ is an isomorphism, $Q_0$ satisfies the formula $\exists y\varphi(y,\bar{a})$. Hence, again by elementarity, $H(\kappa)\models\exists y\varphi(y,\bar{a})$, and since $\bar{a}$ is a tuple of elements in  $P$, by the Tarski-Vaught test there is $b\in P$ such that $H(\kappa)\models\varphi(b,\bar{a})$. Now it is easy to see with a similar argument, using elementarity of the models $Q_0$ and $Q_1$ and the isomorphism $\Psi$, that $H(\kappa)\models\varphi\big(\Psi(b),\Psi(\bar{a})\big)$. Therefore, by the Tarski-Vaught test we can conclude that $\Psi(P)$ is an elementary submodel of $(H(\kappa);\in,T)$.
	\end{proof}
	
	The next result, which was already noticed by Neeman in \cite{Neeman2014Forcingwithsequencesofmodelsoftwotypes}, is very fundamental in the understanding of the structure of side conditions of models of two types and their limitations.
	
	\begin{proposition}\label{propgaps}
		Let $M\in\mathcal{S}$ and $X\in\mathcal{L}\cap M$. Then, for every $Q\in\mathcal{S}\cup\mathcal{L}$ such that $\varepsilon_{X\cap M}\leq\varepsilon_Q<\varepsilon_X$, $Q\notin M$.
	\end{proposition}
	\begin{proof}
		Suppose, towards a contradiction, that there is some $Q\in\mathcal{S}\cup\mathcal{L}$ such that $\varepsilon_{X\cap M}\leq\varepsilon_Q<\varepsilon_X$ and $Q\in M$. First, note that since $X\in\mathcal{L}$, then $\varepsilon_X=X\cap\omega_2$ is an ordinal in $\omega_2$. Therefore, $\varepsilon_Q=\sup(Q\cap\omega_2)$ is an ordinal in $X\cap\omega_2$, and hence $\varepsilon_Q\in M\cap (X\cap\omega_2)$. So we can conclude that $\varepsilon_Q<\varepsilon_{X\cap M}$, which contradicts our assumption.
	\end{proof}
	
	\begin{definition}
		Given a model $Q$, we let
		\[
		Q[\omega_1]:=\{f(\alpha):f\in Q,f\text{ a function with 	}\operatorname{dom}(f)=\omega_1, \alpha\in\omega_1\}.
		\]
		We call $Q[\omega_1]$ the \emph{$\omega_1$-hull of $Q$}, or simply the \emph{hull of $Q$}.
	\end{definition}
	
	The notion of $\omega_1$-hull of a small model plays a central role in the definition of the side conditions that we will introduce in the next section. 
	
	\begin{proposition}\label{prop7}
		Let $Q\in\mathcal{S}\cup\mathcal{L}$. Then $Q[\omega_1]$ is the $\subseteq$-smallest elementary submodel of $H(\kappa)$ that contains $Q\cup\omega_1$ as a subset. 
	\end{proposition}
	\begin{proof}
		First note that since $\omega_1\in Q$, the identity function $id:\omega_1\to\omega_1$ is definable in $Q$, and thus $\omega_1\subseteq Q[\omega_1]$. Moreover, for each $a\in Q$, the constant function sending all $\alpha\in\omega_1$ to $a$ is definable in $Q$, and so $Q\subseteq Q[\omega_1]$. Therefore, $Q\cup\omega_1\subseteq Q[\omega_1]$.
		
		We use the Tarski-Vaught test to show that $Q[\omega_1]$ is an elementary submodel of $H(\kappa)$. Let $\varphi(y,x_0,\dots,x_n)$ be a first-order formula in the language of set theory. Let $a_0,\dots,a_n\in Q[\omega_1]$ be such that $H(\kappa)\models\exists y\varphi(y,a_0,\dots,a_n)$. We have to find $b\in Q[\omega_1]$ such that $H(\kappa)\models\varphi(b,a_0,\dots,a_n)$. Let $f_i\in Q$ and $\alpha_i\in\omega_1$ be such that $a_i=f_i(\alpha_i)$, for all $i\leq n$. Fix a canonically defined bijection $F:\omega_1^{<\omega}\to\omega_1$, and let us note that $F\in Q$. We can define a function $g$ in $H(\kappa)$ by
		\[
		g\big(F(\beta_0,\dots,\beta_n)\big)=d\iff H(\kappa)\models\varphi\big(d,f_0(\beta_0),\dots,f_n(\beta_n)\big),
		\]
		for any $\beta_0,\dots,\beta_n\in\omega_1$. Note that $g\circ F$ is defined with $f_0,\dots,f_n$ and $\omega_1$ as parameters. Hence, 
		$g\circ F\in Q$. Now, let $b=g\big(F(\alpha_0,\dots,\alpha_n)\big)$, and let us note that $b\in Q$ and $H(\kappa)\models \varphi(b,a_0,\dots,a_n)$. Therefore, $Q[\omega_1]\preceq H(\kappa)$.
		
		For the minimality of $Q[\omega_1]$ we observe
		that if $P$ is an elementary submodel of $H(\kappa)$ such that $Q\cup\omega_1\subseteq P$, then $Q[\omega_1]\subseteq P$.
	\end{proof}
	
	\begin{corollary}
		If $X$ is a large model, then $X[\omega_1]=X$.
	\end{corollary}
	
	\begin{definition}
		Let $Q_0,Q_1\in\mathcal{S}\cup\mathcal{L}$. A map $\Psi$ between $Q_0[\omega_1]$ and $Q_1[\omega_1]$ is an \emph{$\omega_1$-isomorphism} if $\Psi$ is the unique isomorphism between the structures $$(Q_0[\omega_1];\in,Q_0,T)$$ and $$(Q_1[\omega_1];\in,Q_1,T),$$ and $\Psi$ is the identity on $Q_0[\omega_1]\cap Q_1[\omega_1]$.
		
		If $\Psi$ is an $\omega_1$-isomorphim between $Q_0[\omega_1]$ and $Q_1[\omega_1]$, we will say that $Q_0$ and $Q_1$ are \emph{$\omega_1$-isomorphic} and we will denote this by $Q_0[\omega_1]\cong Q_1[\omega_1]$. Moreover, we will denote the $\omega_1$-isomorphism $\Psi$ by $\Psi_{Q_0[\omega_1],Q_1[\omega_1]}$.
		
		If $X_0,X_1\in\mathcal{L}$ and $X_0[\omega_1]\cong X_1[\omega_1]$, we will simply write $X_0\cong X_1$.
	\end{definition}
	
	\begin{proposition}\label{prop-iso}
		Let $M_0,M_1\in\mathcal{S}$ and suppose that $M_0[\omega_1]\cong M_1[\omega_1]$. Then $\Psi_{M_0[\omega_1],M_1[\omega_1]}\upharpoonright M_0$ is the unique isomorphism between the structures $(M_0;\in,T)$ and $(M_1;\in,T)$. 
	\end{proposition}
	
	
	
	
	\begin{proposition}
		If $M_0,M_1\in\mathcal{S}$ are such that $M_0[\omega_1]\cong M_1[\omega_1]$, then $\varepsilon_{M_0}=\varepsilon_{M_1}$.
	\end{proposition}
	
	\begin{proposition}\label{prop21}
		Let $X\in\mathcal{L}$ and $M_0,M_1\in\mathcal{S}$. Suppose that $X$ is a member of $M_0[\omega_1]\cap M_1[\omega_1]$ and that $M_0[\omega_1]\cong M_1[\omega_1]$. Then $X\cap M_0=X\cap M_1$.
	\end{proposition}
	\begin{proof}
		Since $X\in M_0[\omega_1]\cap M_1[\omega_1]$, the model $X$ is fixed by $\Psi_{M_0[\omega_1],M_1[\omega_1]}$. Therefore, as $X\cap M_0\in X$, and thus $X\cap M_0\in M_0[\omega_1]\cap M_1[\omega_1]$, we have $X\cap M_0=\Psi_{M_0[\omega_1],M_1[\omega_1]}(X\cap M_0)=X\cap M_1$.
	\end{proof}	
	
	We finish this section with some results that show that, in our setting, the operations of taking intersections and taking isomorphic copies of small and large elementary submodels commute. For a proof of the following facts, see \cite{Gallart2026forcingsymmetricsystemsmodels} (Propositions 2.32-2.35).
	
	\begin{proposition}\label{prop20}
		Let $Q_0, Q_1,Q_1'\in\mathcal{S}\cup\mathcal{L}$ be such that $Q_0\in Q_1[\omega_1]$ and $Q_1[\omega_1]\cong Q_1'[\omega_1]$. Let $Q_0'$ denote $\Psi_{Q_1[\omega_1],Q_1'[\omega_1]}(Q_0)$. Then $$Q_0'[\omega_1]=\Psi_{Q_1[\omega_1],Q_1'[\omega_1]}(Q_0[\omega_1]),$$ and $\Psi_{Q_1[\omega_1],Q_1'[\omega_1]}\upharpoonright Q_0[\omega_1]$ witnesses $Q_0[\omega_1]\cong Q_0'[\omega_1]$.
	\end{proposition}
	
	
	\begin{corollary}\label{corollary8}
		Let $M_0,M_1\in\mathcal{S}$ be such that $M_0[\omega_1]\cong M_1[\omega_1]$. Suppose that $X_0\in\mathcal{L}\cap M_0$ and denote $\Psi_{M_0[\omega_1],M_1[\omega_1]}(X_0)$ by $X_1$. Then, the map $\Psi_{M_0[\omega_1],M_1[\omega_1]}\upharpoonright(X_0\cap M_0)[\omega_1]$ witnesses $(X_0\cap M_0)[\omega_1]\cong(X_1\cap M_1)[\omega_1]$.
	\end{corollary}
	
	\begin{proposition}\label{prop9}
		Let $M\in\mathcal{S}$ and $X_0,X_1\in\mathcal{L}\cap M$. Suppose that $X_0\cong X_1$. Then, $\Psi_{X_0,X_1}\upharpoonright(X_0\cap M)[\omega_1]$ witnesses $(X_0\cap M)[\omega_1]\cong(X_1\cap M)[\omega_1]$.
	\end{proposition}
	
	\begin{proposition}\label{prop10}
		Let $M_0,M_1\in\mathcal{S}$ be such that $M_0[\omega_1]\cong M_1[\omega_1]$. Let $X_0,X_1\in\mathcal{L}$ be such that $X_0\cong X_1$, $X_0\in M_0$ and $X_1\in M_1$. Then, the map $\Psi_{X_0,X_1}\upharpoonright(X_0\cap M_0)[\omega_1]$ witnesses $(X_0\cap M_0)[\omega_1]\cong(X_1\cap M_1)[\omega_1]$.
	\end{proposition}

	\section{Two-type symmetric systems and amalgamation lemmas}\label{section-pure}
	
	In this section, we will introduce $(\mathcal{S},\mathcal{L})$-symmetric systems and outline the main amalgamation lemmas from \cite{Gallart2026forcingsymmetricsystemsmodels}. We will not give a full account of the proofs for such lemmas. We refer the interested reader to the aforementioned paper. 
	
	The side conditions that we will present in this section are a natural combination of Neeman's chains of models of two types from \cite{Neeman2014Forcingwithsequencesofmodelsoftwotypes} and Asperó-Mota-Todor\v cevi\'c's symmetric systems of elementary submodels from \cite{Aspero:Mota2015a} and \cite{Todorcevic1984AnoteonthePFA}. The structural properties of these side conditions, which will be reflected in the 
	amalgamation lemmas, will ensure that the forcing witnessing Theorem \ref{mainthm} is proper, $\aleph_1$-proper and has the $\aleph_3$-Knaster condition (this crucially uses the assumption $2^{\aleph_1}=\aleph_2$).
	
	For any set $\mathcal{M}\subseteq\mathcal{S}\cup\mathcal{L}$, we let $\dom(\mathcal{M})$ denote the set of $\omega_2$-heights $\{\varepsilon_Q:Q\in\mathcal{M}\}$. If $\varepsilon\in\dom(\mathcal{M})$, let $\mathcal{M}(\varepsilon)$ be the set of all models $Q\in\mathcal{M}$ such that $\varepsilon_Q=\varepsilon$. Moreover, we let $\mathcal{M}({<}\varepsilon)$ denote the set of models $Q\in\mathcal{M}$ such that $\varepsilon_Q<\varepsilon$. We define  $\mathcal{M}({\leq}\varepsilon)$, $\mathcal{M}({>}\varepsilon)$, and $\mathcal{M}(\geq\varepsilon)$ similarly. Lastly, we let $\mathcal{M}[\omega_1]$ denote the set $\{Q[\omega_1]:Q\in\mathcal{M}\}$.
	
	\begin{definition}\label{defSSM2T}
		Let $\mathcal{M}$ be a finite set of members of $H(\kappa)$. We say that $\mathcal{M}$ is an \emph{$(\mathcal{S},\mathcal{L})$-symmetric system} if and only if the following holds:
		\begin{enumerate}[label=(\Alph*)]
			\item Every $Q\in\mathcal{M}$ is an element of $\mathcal{S}\cup\mathcal{L}$.
			
			\item For any two $Q_0,Q_1\in\mathcal{M}$, if $\varepsilon_{Q_0}=\varepsilon_{Q_1}$, then $Q_0[\omega_1]\cong Q_1[\omega_1]$. 
			
			\item If $\varepsilon_1$ is the immediate successor of $\varepsilon_0$ in $\dom(\mathcal{M})$ and $Q_0\in\mathcal{M}(\varepsilon_0)$, then there is $Q_1\in\mathcal{M}(\varepsilon_1)$ such that $Q_0\in Q_1$.
			
			\item For all $Q_0,Q_1,Q_1'\in\mathcal{M}$ such that $Q_0\in Q_1$ and $\varepsilon_{Q_1}=\varepsilon_{Q_1'}$, $$\Psi_{Q_1[\omega_1],Q_1'[\omega_1]}(Q_0)\in\mathcal{M}.$$
			
			\item For every $X\in\mathcal{M}\cap\mathcal{L}$ and every $M\in\mathcal{M}\cap\mathcal{S}$, if $X\in M$, then $X\cap M\in\mathcal{M}$.
		\end{enumerate}
	\end{definition}

	\begin{remark}
		If $\mathcal{M}$ is an $(\mathcal{S},\mathcal{L})$-symmetric system and $\varepsilon_0<\varepsilon_1$ belong to $\dom(\mathcal{M})$, then for every $Q_0\in\mathcal{M}(\varepsilon_0)$ we can find a model $Q_1\in\mathcal{M}(\varepsilon_1)$ such that $Q_0\in Q_1[\omega_1]$ by successive applications of clause (C).
	\end{remark}
	
	We will refer to clause (C) and the conclusion of the last remark, indistinctly, as the \emph{shoulder axiom for $\mathcal{M}$}. Moreover, when we talk about the \emph{symmetry} of an $(\mathcal{S},\mathcal{L})$-symmetric system, we usually refer to the closure under isomorphisms given by the following strengthening of condition (D).
	
	\begin{proposition}[Proposition 3.4 in \cite{Gallart2026forcingsymmetricsystemsmodels}]\label{prop6}
		Let $\mathcal{M}$ be an $(\mathcal{S},\mathcal{L})$-symmetric system, and let $Q_0,Q_1,Q_1'$ be in $\mathcal{M}$. If $Q_0\in Q_1[\omega_1]$ and $\varepsilon_{Q_1}=\varepsilon_{Q_1'}$, then $\Psi_{Q_1[\omega_1],Q_1'[\omega_1]}(Q_0)\in\mathcal{M}$.
	\end{proposition}
	
	The following are some basic facts about $(\mathcal{S},\mathcal{L})$-symmetric systems, which are easy consequences of the preliminary material on elementary submodels from Section \ref{section-prelim}. We omit their proofs, which can be found in \cite{Gallart2026forcingsymmetricsystemsmodels}.
	
	\begin{proposition}[Proposition 3.3 in \cite{Gallart2026forcingsymmetricsystemsmodels}]\label{prop18}
		Let $\mathcal{M}$ be an $(\mathcal{S},\mathcal{L})$-symmetric system. The following hold for any two models $Q_0,Q_1\in\mathcal{M}$:
		\begin{enumerate}[label=(\arabic*)]
			\item If $Q_0\in Q_1$ and $|Q_0|\leq|Q_1|$, then $Q_0\subseteq Q_1$.
			\item If $Q_1$ is a small model, $Q_0\in Q_1[\omega_1]$, and there is no large model $X\in\mathcal{M}$ such that $\varepsilon_{Q_0}<\varepsilon_X<\varepsilon_{Q_1}$, then $Q_0\in Q_1$.
		\end{enumerate}
	\end{proposition}
	
	\begin{proposition}[Proposition 3.5 in \cite{Gallart2026forcingsymmetricsystemsmodels}]\label{prop13}
		Let $\mathcal{M}$ be an $(\mathcal{S},\mathcal{L})$-symmetric system and let $Q_0,Q_1$ be two models in $\mathcal{M}$ such that $Q_0\in Q_1[\omega_1]$. If there is $\varepsilon\in\dom(\mathcal{M})$ such that $\varepsilon_{Q_0}<\varepsilon<\varepsilon_{Q_1}$, then there is some $P\in\mathcal{M}(\varepsilon)$ such that $Q_0\in P[\omega_1]$ and $P\in Q_1[\omega_1]$.
	\end{proposition}
	
	\begin{proposition}[Proposition 3.6 in \cite{Gallart2026forcingsymmetricsystemsmodels}]\label{isocopy}
		Let $\mathcal{M}$ be an $(\mathcal{S},\mathcal{L})$-symmetric system. Let $Q_0,Q_1$ be two elementary submodels of $H(\kappa)$ such that $\Psi_{Q_0,Q_1}$ is the unique isomorphism between $(Q_0;\in)$ and $(Q_1;\in)$. If $\mathcal{M}\subseteq Q_0$, then $\Psi_{Q_0,Q_1}(\mathcal{M})$ is an $(\mathcal{S},\mathcal{L})$-symmetric system.
	\end{proposition}
	
	Next, we will include some results, which are standard in the context of forcing with side conditions, for amalgamating  $(\mathcal{S},\mathcal{L})$-symmetric systems. The preservation properties of the pure side condition forcing follow, by standard arguments, from these amalgamation lemmas, but most importantly, as we shall see in the next section, the preservation theorems for the forcing witnessing Theorem \ref{mainthm} rely very heavily on the following results. We will omit the proofs and only give a sketch of the main ideas.
	
	The first amalgamation lemma, which appears as Lemma 3.7 in \cite{Gallart2026forcingsymmetricsystemsmodels}, asserts that any $(\mathcal{S},\mathcal{L})$-symmetric system can be extended by adding any model that contains it.
	
	\begin{lemma}\label{amal-ontop}
		Let $\mathcal{M}$ be an $(\mathcal{S},\mathcal{L})$-symmetric system and let $Q\in\mathcal{S}\cup\mathcal{L}$ be such that $\mathcal{M}\subseteq Q$. Then there is an $(\mathcal{S},\mathcal{L})$-symmetric system $\mathcal{M}_Q$ such that $\mathcal{M}\cup\{Q\}\subseteq\mathcal{M}_Q$ 
	\end{lemma}

	If $Q$ is a large model, it is easy to check that $\mathcal{M}_Q=\mathcal{M}\cup\{Q\}$ is an $(\mathcal{S},\mathcal{L})$-symmetric system. If $Q$ is a small model, then $\mathcal{M}_Q$ can be defined to be the closure of $\mathcal{M}\cup\{Q\}$ under the relevant intersections. That is, the $(\mathcal{S},\mathcal{L})$-symmetric system $\mathcal{M}_Q$ witnessing the conclusion of the last lemma is
	\[
	\mathcal{M}_Q=\mathcal{M}\cup\{Q\}\cup\{X\cap Q:Q\in\mathcal{M}\cap\mathcal{L}\}.
	\]
	
	The next lemma asserts that any $(\mathcal{S},\mathcal{L})$-symmetric system can be naturally reflected inside any of its models.
	
	\begin{lemma}\label{amal-rest}
		Let $\mathcal{M}$ be an $(\mathcal{S},\mathcal{L})$-symmetric system and let $Q$ be a member of $\mathcal{M}\cup\mathcal{M}[\omega_1]$. Then $\mathcal{M}\cap Q$ is an $(\mathcal{S},\mathcal{L})$-symmetric system.
	\end{lemma}

	The proof of the second amalgamation lemma, which appears as Lemma Lemma 3.10 in \cite{Gallart2026forcingsymmetricsystemsmodels}, is straightforward if $Q\in\mathcal{L}$ or $Q\in\mathcal{M}[\omega_1]$. Indeed, in that case $\mathcal{M}\cap Q$ coincides with the set of all models in $\mathcal{M}$ lying below $Q$,\footnote{A model $P$ is said to lie below $Q$ if $P\in Q[\omega_1]$.} and the verification of clauses (A)-(E) from Definition \ref{defSSM2T} is an easy exercise. However, if $Q$ is a small model, there may be models in $\mathcal{M}$ lying below $Q$ that are not members of $Q$. For instance, according to Proposition \ref{propgaps}, every model $P\in\mathcal{M}$ such that $\varepsilon_{X\cap Q}\leq\varepsilon_P<\varepsilon_X$, for some $X\in\mathcal{M}\cap Q\cap\mathcal{L}$, fails to be a member of $Q$. In this case, the proof requires a previous analysis of the structure of $\dom(\mathcal{M}\cap Q)$, which is given in Proposition 3.9 of that same paper.
	We will give more details about the structure of $\mathcal{M}\cap Q$ in the outline of the proof of Lemma \ref{lemma-amal}.

	\begin{notation}
		If $\mathcal{M}$ is an $(\mathcal{S},\mathcal{L})$-symmetric system, recall that we let $\mathcal{M}[\omega_1]$ denote the set $\{Q[\omega_1]:Q\in\mathcal{M}\}$. If $\mathcal{M}_0$ and $\mathcal{M}_1$ are two $(\mathcal{S},\mathcal{L})$-symmetric systems, we will write $\mathcal{M}_0\cong\mathcal{M}_1$ to denote that, for some $n<\omega$, there are enumerations $\langle Q_i^0:i<n\rangle$ and $\langle Q_i^1:i<n\rangle$ of $\mathcal{M}_0$ and $\mathcal{M}_1$, respectively, such that the structures $(\bigcup\mathcal{M}_0[\omega_1];\in,Q_i^0)_{i<n}$ and $(\bigcup\mathcal{M}_1[\omega_1];\in,Q_i^1)_{i<n}$ are isomorphic via an isomorphism which is the identity on $\bigcup\mathcal{M}_0[\omega_1]\cap\bigcup\mathcal{M}_1[\omega_1]$.
	\end{notation}
	
	The next amalgamation lemma will be crucially used to ensure that the forcing from Section \ref{section-strong-chains}, witnessing our main theorem (Theorem \ref{mainthm}), has the $\aleph_3$-Knaster condition, provided that $2^{\aleph_1}=\aleph_2$ holds in the ground model. In \cite{Gallart2026forcingsymmetricsystemsmodels}, it appears as Lemma 3.12, and the proof is standard in the context of forcing with symmetric systems as side conditions.
	
	\begin{lemma}\label{pureamalgamation2}
		Let $n<\omega$ and suppose that $\mathcal{M}_0,\dots,\mathcal{M}_n$ are $(\mathcal{S},\mathcal{L})$-sym\-me\-tric systems such that $\mathcal{M}_i\cong\mathcal{M}_j$, for all $i,j\leq n$. Then $\bigcup_{i\leq n}\mathcal{M}_i$ is an $(\mathcal{S},\mathcal{L})$-symmetric system.
	\end{lemma}

	The next (and last) amalgamation lemma gives one of the most important properties of the side conditions. In the context of the pure side condition forcing, the preservation of the cardinals $\aleph_1$ and $\aleph_2$ is an immediate consequence of the next result (see Theorems 3.23 and 3.24 from \cite{Gallart2026forcingsymmetricsystemsmodels}). In the case of the forcing from Section \ref{section-strong-chains}, the preservation of cardinals will require a bit more work and the next lemma will play a central role in those arguments. 
	
	In \cite{Gallart2026forcingsymmetricsystemsmodels}, the next result is divided into Lemmas 3.19 and 3.20, which correspond to the cases $Q\in\mathcal{L}$ and $Q\in\mathcal{S}$, respectively.

	\begin{lemma}\label{lemma-amal}
		Let $\mathcal{M}$ be an $(\mathcal{S},\mathcal{L})$-symmetric system, and let $Q\in\mathcal{M}$. Let $\mathcal{W}$ be another $(\mathcal{S},\mathcal{L})$-symmetric system such that $\mathcal{M}\cap Q\subseteq\mathcal{W}\subseteq Q$. Then there is an $(\mathcal{S},\mathcal{L})$-symmetric system $\mathcal{U}$ such that $\mathcal{M}\cup\mathcal{W}\subseteq\mathcal{U}$. 
	\end{lemma}
	
	The idea behind the proof of the above lemma is that the $(\mathcal{S},\mathcal{L})$-symmetric system $\mathcal{U}$ witnessing the compatibility of $\mathcal{M}$ and $\mathcal{W}$ is the result of closing $\mathcal{M}\cup\mathcal{W}$ under both intersections and $\omega_1$-isomorphisms. However, this needs to be done in a very surgical way if one wants to avoid adding too many models, which makes it somewhat
	 hard to check whether the resulting system is indeed an $(\mathcal{S},\mathcal{L})$-symmetric system. Propositions 3.14-3.18 from the aforementioned paper were isolated precisely to capture this surgical approach to the closure of $\mathcal{M}\cup\mathcal{W}$ under intersections and $\omega_1$-isomorphisms. In a very loose way, these propositions assert that it is enough to close a small fragment of $\mathcal{M}\cup\mathcal{W}$ under intersections so that the subsequent closure under $\omega_1$-isomorphisms results in an $(\mathcal{S},\mathcal{L})$-symmetric system. They rely heavily on the fact that, as Propositions \ref{prop20}-\ref{prop10} show, in the context of $(\mathcal{S},\mathcal{L})$-symmetric systems, the operations of taking intersections and $\omega_1$-isomorphic copies commute. We will not state them explicitly, but their content will be included in
	  the outline of the proof of Lemma \ref{lemma-amal}. 
	
	If $Q$ is a large model, the conclusion of the last amalgamation lemma follows from two important properties of $(\mathcal{S},\mathcal{L})$-symmetric systems, which appear in \cite{Gallart2026forcingsymmetricsystemsmodels} as Propositions 3.14 and 3.18. The first one is that, since $\mathcal{M}\cap Q$ contains every model in $\mathcal{M}$ lying below $Q$, as it was observed in the outline of the proof of Lemma \ref{amal-rest} above, $\mathcal{M}\cup\mathcal{W}$ is already closed under the relevant intersections. That is, and this is Proposition 3.14 in \cite{Gallart2026forcingsymmetricsystemsmodels}, it can be shown that every model of the form $X\cap M$, where $M\in\mathcal{M}\cap\mathcal{S}$ and $X\in\mathcal{W}\cap\mathcal{L}\cap M$, is a member of $\mathcal{W}$. The second property is that, if $\mathcal{M}\cup\mathcal{W}$ is closed under intersections in the above sense, then the closure of $\mathcal{M}\cup\mathcal{W}$ under the relevant $\omega_1$-isomorphisms is an $(\mathcal{S},\mathcal{L})$-symmetric system. More precisely, and this is Proposition 3.18 from the same paper,
	\[
	\mathcal{U}:=\mathcal{M}(\geq\varepsilon_Q)\cup\{\Psi_{Q[\omega_1],P[\omega_1]}(W):W\in\mathcal{W},P\in\mathcal{M}(\varepsilon_Q)\}
	\]
	is an $(\mathcal{S},\mathcal{L})$-symmetric system that extends $\mathcal{M}\cup\mathcal{W}$. 
	
	However, if $Q$ is a small model, the argument requires a much deeper analysis of the structure of $\mathcal{M}\cup\mathcal{W}$. It is no longer the case that $\mathcal{M}\cup\mathcal{W}$ is closed under the relevant intersections, and this is due to the fact that, as already mentioned in the sketch of the proof of Lemma \ref{amal-rest}, $\mathcal{M}\cap Q$ may not contain all the models in $\mathcal{M}$ lying below $Q$. However, this is not the only obstacle in the construction of an $(\mathcal{S},\mathcal{L})$-symmetric system extending both $\mathcal{M}$ and $\mathcal{W}$. Let us first describe the structure of $\mathcal{M}\cap Q$ and how the models in $\mathcal{W}$ sit below $Q$, and then we will give an overview of all the obstacles that must be kept in mind when building $\mathcal{U}$. 
	
	Fix a maximal $\in$-chain $X_0\in\dots\in X_{n-1}$ of large models in $\mathcal{M}\cap Q$, and note that for every $i<n$, $X_i\in X_{i+1}\cap Q$. Hence, we have the following configuration:
	\[
	X_0\cap Q\in X_0\in\dots\in X_i\cap Q\in X_i\in X_{i+1}\cap Q\in\dots\in X_{n-1}\in Q.
	\] 
	Let us divide $\dom(\mathcal{M}\cap Q[\omega_1])\cup\dom(\mathcal{W})$ into two disjoint sets of ordinals that we will denote by $B$ and $D$. Let $B$ be the union of the following intervals of ordinals
	\begin{itemize}
		\item $B_0:=[\varepsilon_{\min},\varepsilon_{X_0\cap Q})$, where $\varepsilon_{\min}:=\min\big(\dom(\mathcal{M}\cap Q[\omega_1])\cup\dom(\mathcal{W})\big)$,
		\item $B_{i+1}:=[\varepsilon_{X_i},\varepsilon_{X_{i+1}\cap Q})$, for every $i<n-1$, and
		\item $B_n:=[\varepsilon_{X_{n-1}},\varepsilon_Q)$,
	\end{itemize}
	and let $D$ be the union of the following intervals of ordinals
	\begin{itemize}
		\item $D_i:=[\varepsilon_{X_i\cap Q},\varepsilon_{X_i})$, for every $i<n$.
	\end{itemize}
	If $Q\in\mathcal{S}\cup\mathcal{L}$, we will say that $Q$ is in the \emph{bright area} if $\varepsilon_Q\in B$, and that $Q$ is in the \emph{dark area} if $\varepsilon_Q\in D$.
	
	\begin{remark}\label{pure-amal-remark1}
		Since $\langle X_i:i<n\rangle$ is a maximal $\in$-chain of large models in $\mathcal{M}\cap Q$, if $P\in\mathcal{M}\cap Q[\omega_1]$ is a model in the bright area, then $P$ is either a small model or satisfies $\varepsilon_P=\varepsilon_{X_i}$ for some $i<n$. 
	\end{remark}
	
	\begin{remark}\label{pure-amal-remark2}
		Proposition \ref{propgaps} implies that all models $P\in\mathcal{M}\cap Q[\omega_1]$ in the dark area are not members of $Q$. 
	\end{remark}
	
	\begin{remark}\label{pure-amal-remark3}
		All the models in $\mathcal{W}$ (and hence in $\mathcal{M}\cap Q$) belong to the bright area.
	\end{remark}
	
	\begin{remark}\label{pure-amal-remark4}
		It is not in general true that every model $P\in\mathcal{M}\cap Q[\omega_1]$ in the bright area is a member of $Q$. Indeed, suppose that $\varepsilon_P\in B_{i+1}$, for some $i<n-1$. By Proposition \ref{prop13}, there must be a model $N_{i+1}\in\mathcal{M}\cap Q[\omega_1]$ such that $\varepsilon_{N_{i+1}}=\varepsilon_{X_{i+1}\cap Q}$ and $P\in N_{i+1}[\omega_1]$. Since all models with $\omega_2$-height in the interval $B_{i+1}\setminus\{\varepsilon_{X_i}\}$ are small, Proposition \ref{prop18} ensures that $P\in N_{i+1}$. However, $N_{i+1}$ need not be of the form $Y_{i+1}\cap Q$ for some $Y_{i+1}\in\mathcal{M}\cap Q\cap\mathcal{L}$ such that $\varepsilon_{Y_{i+1}}=\varepsilon_{X_{i+1}}$. Hence, there is no guarantee that $P$ is a member of $Q$.
	\end{remark}
	
	Now that we have described the structure of $\mathcal{M}\cap Q$ and how the models in $\mathcal{W}$ are placed within, we are ready to start the construction of the $(\mathcal{S},\mathcal{L})$-symmetric system $\mathcal{U}$. As anticipated, $\mathcal{U}$ will be the closure of $\mathcal{M}\cup\mathcal{W}$ under intersections and $\omega_1$-isomorphisms. However, we want to add just enough models to secure this closure, since otherwise it becomes very hard to check whether the resulting system is indeed an $(\mathcal{S},\mathcal{L})$-symmetric system. Hence, let us start by enumerating the three sets of models that any $(\mathcal{S},\mathcal{L})$-symmetric system extending $\mathcal{M}\cup\mathcal{W}$ is missing:
	\begin{enumerate}[label=(\Alph*)]
		\item Let $W\in (\mathcal{W}\setminus\mathcal{M})\cap\mathcal{L}$ be such that $\varepsilon_W=\varepsilon_{X_i}$, for some $i<n$. Let $P\in\mathcal{M}\cap X_i$ be such that $\varepsilon_P\in D_i$. Since $P$ is in the dark area and is therefore not a member of $Q$, it is not a member of $\mathcal{W}$ either. Hence, the image of the model $P$ under the isomorphism  $\Psi_{X_i,W}$ may not be an element of $\mathcal{M}\cup\mathcal{W}$. 
		
		\item Let $W\in (\mathcal{W}\setminus\mathcal{M})\cap\mathcal{L}$ such that $\varepsilon_W\neq\varepsilon_{X_i}$ for any $i<n$. Let $M$ be a small model in $\mathcal{M}\setminus Q$ such that $W\in M$. Then $W\cap M$ may not be a member of $\mathcal{M}\cup\mathcal{W}$.
		
		\item Let $N_i$ be a small model in $\mathcal{M}\setminus Q$ such that $\varepsilon_{N_i}=\varepsilon_{X_i\cap Q}$. Let $W\in\mathcal{W}\setminus\mathcal{M}$ be such that $W\in Y_i\cap Q$, where $Y_i\in\mathcal{M}(\varepsilon_{X_i})\cap Q$. Then $\Psi_{(X_i\cap Q)[\omega_1],N_i[\omega_1]}(W)$ may not be a member of $\mathcal{W}\cup\mathcal{M}$. This is a consequence of Remark \ref{pure-amal-remark4}.
	\end{enumerate}
	
	Adding to $\mathcal{M}\cup\mathcal{W}$ all the missing models from (A)-(C) will not secure the closure of $\mathcal{M}\cup\mathcal{W}$ under intersections and $\omega_1$-isomorphisms. Indeed, if we add a model $W\cap M$, as in (B), to $\mathcal{M}\cup\mathcal{W}$, there is no guarantee that the image of $W\cap M$ under an isomorphism $\Psi_{(X_i\cap Q)[\omega_1],N_i[\omega_1]}$ as in (C) is in $\mathcal{M}\cup\mathcal{W}$. Similarly, if we add a model $\Psi_{(X_i\cap Q)[\omega_1],N_i[\omega_1]}(W)$, as in (C), to $\mathcal{M}\cup\mathcal{W}$, there is no guarantee that if $M$ is a small model in $\mathcal{M}\cup\mathcal{W}$ such that $\Psi_{(X_i\cap Q)[\omega_1],N_i[\omega_1]}(W)\in M$, then the intersection of $\Psi_{(X_i\cap Q)[\omega_1],N_i[\omega_1]}(W)$ and $M$ is in $\mathcal{M}\cup\mathcal{W}$. 
	
	This problem can be addressed by arguing inductively over ``initial segments" of $\mathcal{M}\cup\mathcal{W}$. Thanks to Propositions \ref{prop20}-\ref{prop10}, the operations of taking intersections and $\omega_1$-isomorphic copies commute enough in $(\mathcal{M}\cup\mathcal{W})\cap X_i$, for every $i<n$, to guarantee that after first adding the models of the form $W\cap M$, for small $M\in(\mathcal{M}\setminus Q)\cap X_i$ and large $W\in(\mathcal{W}\setminus\mathcal{M})\cap X_i$ such that $W\in M$, and then closing the resulting system under $\omega_1$-isomorphisms, we produce an $(\mathcal{S},\mathcal{L})$-symmetric system extending $(\mathcal{M}\cup\mathcal{W})\cap X_i$. This is captured in Propositions 3.14-3.18 from \cite{Gallart2026forcingsymmetricsystemsmodels}.
	
	Let us be more precise. For every $i<n$, denote
	\begin{itemize}
		\item $\mathcal{M}_i:=\mathcal{M}\cap X_i$,
		\item $\mathcal{W}_i:=\mathcal{W}\cap X_i=\mathcal{W}\cap(X_i\cap Q)$ (note that $\mathcal{W}_i\subseteq X_i\cap Q$), and
		\item $Q_i:=X_i\cap Q$.
	\end{itemize}   
	Suppose that for some $i<n-1$, $\mathcal{U}_i$ is an $(\mathcal{S},\mathcal{L})$-symmetric system extending $(\mathcal{M}\cup\mathcal{W})\cap X_i$. Define $\mathcal{V}_{i+1}$ as the result of adding all $\omega_1$-isomorphic copies of $\mathcal{U}_i$ to $(\mathcal{M}\cup\mathcal{W})\cap Q_{i+1}[\omega_1]$. Equivalently, let $\mathcal{V}_{i+1}$ be the set 
	\[
	\mathcal{W}_{i+1}(\geq\varepsilon_{X_i})\cup\{\Psi_{X_i,Y_i}(U):U\in\mathcal{U}_i,Y_i\in\mathcal{W}_{i+1}(\varepsilon_{X_i})\}.
	\]
	Note that $\mathcal{V}_{i+1}$ is contained in $Q_{i+1}[\omega_1]$. Proposition 3.18 from \cite{Gallart2026forcingsymmetricsystemsmodels} ensures that $\mathcal{V}_{i+1}$ is an $(\mathcal{S},\mathcal{L})$-symmetric system. This first step deals with the models of the form (A) above. The next step is closing $$\big((\mathcal{M}\cup\mathcal{W})\cap X_{i+1}\big)\cup\mathcal{V}_{i+1}$$ under intersections. Let $\mathcal{V}_{i+1}^*$ be the result of adding to $\mathcal{V}_{i+1}$ every model of the form $W\cap M$, where $M$ is a small model in $\mathcal{M}_{i+1}(\geq\varepsilon_{Q_{i+1}})$ and $W$ is a large model in $\mathcal{W}_{i+1}(\geq\varepsilon_{X_i})\cap M$. Note that $\mathcal{V}_{i+1}^*$ is still contained in $Q_{i+1}[\omega_1]$. The main part of the argument is showing that $\mathcal{V}_{i+1}^*$ is an $(\mathcal{S},\mathcal{L})$-symmetric system with the property that $\big((\mathcal{M}\cup\mathcal{W})\cap X_{i+1}\big)\cup\mathcal{V}_{i+1}^*$ is closed under intersections. This second step deals with the models of the form (B) above. The last step is closing $\big((\mathcal{M}\cup\mathcal{W})\cap X_{i+1}\big)\cup\mathcal{V}_{i+1}^*$ under $\omega_1$-isomorphisms. Define $\mathcal{U}_{i+1}$ as the result of adding all $\omega_1$-isomorphic copies of $\mathcal{V}_{i+1}^*$ to $(\mathcal{M}\cup\mathcal{W})\cap X_{i+1}$. That is, let $\mathcal{U}_{i+1}$ be the set
	\[
	\mathcal{M}_{i+1}(\geq\varepsilon_{Q_{i+1}})\cup\{\Psi_{Q_{i+1}[\omega_1],N[\omega_1]}(V):V\in\mathcal{V}_{i+1}^*,N\in\mathcal{M}_{i+1}(\varepsilon_{Q_{i+1}})\}.
	\]
	Proposition Proposition 3.18 from \cite{Gallart2026forcingsymmetricsystemsmodels} ensures that $\mathcal{U}_{i+1}$ is an $(\mathcal{S},\mathcal{L})$-symmetric system. In this last step, the commutativity of the operations of taking intersections and $\omega_1$-isomorphic copies, given by Propositions \ref{prop20}-\ref{prop10}, plays a central role in guaranteeing that $\mathcal{U}_{i+1}$ is an $(\mathcal{S},\mathcal{L})$-symmetric system. 
	
	This induction yields an $(\mathcal{S},\mathcal{L})$-symmetric system $\mathcal{U}$ that extends $\mathcal{M}\cup\mathcal{W}$.

	\begin{definition}
		Let $\mathbb{M}(\mathcal{S},\mathcal{L})$ be the poset consisting of $(\mathcal{S},\mathcal{L})$-symmetric systems ordered by reverse inclusion.
	\end{definition}
	
	Let us finish this section by giving a brief overview of the properties of the forcing $\mathbb{M}(\mathcal{S},\mathcal{L})$ and variants thereof.
	
	\begin{corollary}[$2^{\aleph_1}=\aleph_2$]\label{pure-preservation}
		$\mathbb{M}(\mathcal{S},\mathcal{L})$ is strongly $\mathcal{S}$-proper, strongly $\mathcal{L}$-proper and has the $\aleph_3$-Knaster condition. Hence, it preserves all cardinals. Moreover, $\mathbb{M}(\mathcal{S},\mathcal{L})$ preserves $2^{\aleph_1}=\aleph_2$.
	\end{corollary}
	
	The forcing $\mathbb{M}(\mathcal{S},\mathcal{L})$ adds a two-type symmetric system that covers $H(\kappa)^V$ and an $\omega_1$-club subset of $\omega_2$ that avoids all infinite subsets from the ground model. Moreover, in \cite{Gallart2026forcingsymmetricsystemsmodels}, it was also proven that $\mathbb{M}(\mathcal{S},\mathcal{L})$ adds a Kurepa tree on $\omega_2$. Variants of the same forcing add a club subset of $\omega_2$, a function on $\omega_2$ that bounds every canonical function below $\omega_3$ on a club, and a simplified $(\omega_2,1)$-morass. Furthermore, these variants of $\mathbb{M}(\mathcal{S},\mathcal{L})$ retain the preservation properties given by Corollary \ref{pure-preservation}.

	\section{The forcing construction}\label{section-strong-chains}
	
	Let us assume that $\mathrm{GCH}$ holds throughout this section. We start out by fixing a sequence $\vec{e}=(e_\alpha\,:\, \alpha\in \omega_3)$ such that $e_\alpha:|\alpha|\longrightarrow\alpha$ is a bijection for each $\alpha<\omega_3$. Furthermore, we assume that $\kappa=\omega_3$ and $T=\vec{e}$. Hence, $\mathcal{S}$ is the collection of countable elementary submodels $M$ of $(H(\omega_3); \in, \vec{e})$, and we let $\mathcal{L}$ be the collection of all elementary submodels $X$ of $(H(\omega_3); \in, \vec{e})$ such that $|X|=\aleph_1$ and $^\omega X\subseteq X$.
	
	Let us start by proving some general facts about $(\mathcal{S},\mathcal{L})$-symmetric systems that will be crucially used in the proof of Theorem \ref{mainthm}. 
	
	\begin{lemma}\label{agreementiso}
		For all countable $M_0,M_1\preceq(H(\omega_3);\in,\vec{e})$, if $$(M_0[\omega_1];\in,M_0,\vec{e})\cong(M_1[\omega_1];\in,M_1,\vec{e}),$$ then $M_0\cap M_1\cap\omega_3$ is an initial segment of both $M_0\cap\omega_3$ and $M_1\cap\omega_3$. 
	\end{lemma}
	\begin{proof}
		We will show that for every $\beta\in M_0\cap M_1\cap\omega_3$, if $\alpha\in M_0\cap\beta$, then $\alpha\in M_1\cap\beta$. Note that there is some $\xi\in M_0\cap\omega_2$ such that $e_\beta(\xi)=\alpha$ and, in fact, this is seen by $M_0[\omega_1]$. Therefore, 
		\[
		M_1[\omega_1]\models\Psi_{M_0[\omega_1],M_1[\omega_1]}(\xi)\in M_1\cap\omega_2.
		\]
		But note that since $\Psi_{M_0[\omega_1],M_1[\omega_1]}$ is the identity on $M_0[\omega_1]\cap\omega_2$, we have that $\Psi_{M_0[\omega_1],M_1[\omega_1]}(\xi)=\xi\in M_1\cap\omega_2$. Hence, $\alpha=e_\beta(\xi)\in M_1$ as we wanted.
	\end{proof}
	
	\begin{proposition}
		Let $\mathcal{M}$ be an $(\mathcal{S},\mathcal{L})$-symmetric system. If $M,N\in\mathcal{M}\cap\mathcal{S}$ are such that $\delta_M<\delta_N$ and there is no $X\in\mathcal{M}\cap\mathcal{L}$ such that $\varepsilon_M<\varepsilon_X<\varepsilon_N$ or $\varepsilon_N<\varepsilon_X<\varepsilon_M$, then $\varepsilon_M<\varepsilon_N$. 
	\end{proposition}
	\begin{proof}
		If $\varepsilon_M$ and $\varepsilon_N$ were equal, then we would have $\delta_M=\delta_N$. If $\varepsilon_M>\varepsilon_N$, then by the shoulder axiom there would be a model $M'\in\mathcal{M}(\varepsilon_M)$ such that $N\in M'[\omega_1]$, and since there are no large models $X\in\mathcal{M}$ so that $\varepsilon_X$ separates $\varepsilon_M$ and $\varepsilon_N$, we would have $N\in M'$. Therefore, $\delta_N<\delta_{M'}=\delta_M$.
	\end{proof}
	
	\begin{lemma}\label{key-claim}
		Let $\mathcal{M}$ be an $(\mathcal{S},\mathcal{L})$-symmetric system. If $M,N\in\mathcal{M}\cap\mathcal{S}$ are such that $\delta_M<\delta_N$, then $M\cap N\cap\omega_3\in N$.	
	\end{lemma}
	\begin{proof}
		Let $\langle\varepsilon_i:i<n\rangle$ be the increasing enumeration of $\dom(\mathcal{M}\cap\mathcal{L})$, and suppose that $\varepsilon_n<\omega_2$ is an ordinal greater than $\max\dom(\mathcal{M})$. We will show, by induction on $i\leq n$, that if $\varepsilon_M,\varepsilon_N<\varepsilon_i$, then the statement of the lemma holds.
		
		Suppose first that $\varepsilon_M,\varepsilon_N<\varepsilon_0$. Use the shoulder axiom to find $N'\in\mathcal{M}(\varepsilon_N)$ such that $M\in N'[\omega_1]$, and note that since $\varepsilon_M,\varepsilon_N=\varepsilon_{N'}<\varepsilon_0$, we have $M\in N'$. Let now $M':=\Psi_{N'[\omega_1],N[\omega_1]}(M)$, and note that $M'\in \mathcal{M}\cap N$. By Lemma \ref{agreementiso}, $M\cap M'\cap\omega_3$ is an initial segment of $M'\cap\omega_3$. Hence, as $M'\cap\omega_3\in N$, we have $M\cap M'\cap\omega_3\in N$. We will be done by showing that 
		\[
		M\cap N\cap\omega_3=M\cap M'\cap\omega_3.
		\]
		Let $\alpha\in M\cap N\cap\omega_3$, and note that since $M\in N'$, $\alpha\in N'\cap N$. Therefore, 
		\[
		\alpha=\Psi_{N'[\omega_1],N[\omega_1]}(\alpha)\in\Psi_{N'[\omega_1],N[\omega_1]}(M)=M'.
		\]
		This shows that $M\cap N\cap\omega_3\subseteq M'$, and since $M'\subseteq N$, this is enough to get
		\[
		M\cap N\cap\omega_3=M\cap M'\cap\omega_3\in N.
		\]
		
		Assume now that $\varepsilon_M,\varepsilon_N<\varepsilon_{i+1}$ for some $i<n$, and suppose that the statement holds for every pair of models $M_0,N_0\in\mathcal{M}\cap\mathcal{S}$ such that $\varepsilon_{M_0},\varepsilon_{N_0}<\varepsilon_i$. If $\varepsilon_M,\varepsilon_N<\varepsilon_i$, the conclusion follows from the induction hypothesis, and if $\varepsilon_i<\varepsilon_M,\varepsilon_N$, we can argue exactly as we did in the case $i=0$ to get the conclusion. Therefore, we can divide the rest of the proof in the following two cases:
		
		\textbf{Case 1.} Suppose that $\varepsilon_M<\varepsilon_i<\varepsilon_N$. Use the shoulder axiom twice to find $X\in\mathcal{M}(\varepsilon_i)$ and $N'\in\mathcal{M}(\varepsilon_N)$ such that $M\in X\in N'[\omega_1]$. Since $\varepsilon_i=\varepsilon_X<\varepsilon_{N'}<\varepsilon_{i+1}$, we have $X\in N'$. Let $Y:=\Psi_{N'[\omega_1],N[\omega_1]}(X)$, and note that $Y\in \mathcal{M}\cap N$. Therefore, $Y\cap N$ is a member of $\mathcal{M}$. Note that $\varepsilon_{Y\cap N}<\varepsilon_i$ and $\delta_{Y\cap N}=\delta_N$. Hence, by the induction hypothesis,
		\[
		M\cap (Y\cap N)\cap\omega_3\in Y\cap N\subseteq N,
		\] 
		and thus we will be done by showing that $M\cap N\cap\omega_3=M\cap (Y\cap N)\cap\omega_3$. Let $\alpha\in M\cap N\cap\omega_3$. Since $M\in N'[\omega_1]$, we have $\alpha\in N'[\omega_1]\cap N$. Combining this with the fact that $\alpha\in M\subseteq X$, we get
		\[
		\alpha=\Psi_{N'[\omega_1],N[\omega_1]}(\alpha)\in\Psi_{N'[\omega_1],N[\omega_1]}(X)=Y.
		\]
		This shows that $M\cap N\cap\omega_3\subseteq Y\cap N$, and since $Y\cap N\subseteq N$, this is enough to get
		\[
		M\cap N\cap\omega_3=M\cap (Y\cap N)\cap\omega_3\in N.
		\]
		
		\textbf{Case 2.} Suppose that $\varepsilon_N<\varepsilon_i<\varepsilon_M$. By two applications of the shoulder axiom we can find $X'\in\mathcal{M}(\varepsilon_i)$ and $M'\in\mathcal{M}(\varepsilon_M)$ such that $N\in X'\in M'[\omega_1]$. Since $\varepsilon_i=\varepsilon_{X'}<\varepsilon_{M'}<\varepsilon_{i+1}$, we have $X'\in M'$. Let $X:=\Psi_{M'[\omega_1],M[\omega_1]}(X')$, and note that $X\in\mathcal{M}\cap M$. Therefore, $X\cap M$ is a member of $\mathcal{M}$. Note that $\varepsilon_{X\cap M}<\varepsilon_i$ and $\delta_{X\cap M}=\delta_M$. Hence, by the induction hypothesis,
		\[
		(X\cap M)\cap N\cap\omega_3\in N,
		\]
		and thus we will be done by showing that $M\cap N\cap\omega_3=(X\cap M)\cap N\cap\omega_3$. Let $\alpha\in M\cap N\cap\omega_3$. Since $N\in M'[\omega_1]$, we have $\alpha\in M\cap M'[\omega_1]$. Combining this with the fact that $\alpha\in N\subseteq X'$, we get
		\[
		\alpha=\Psi_{M'[\omega_1],M[\omega_1]}(\alpha)\in\Psi_{M'[\omega_1],M[\omega_1]}(X')=X.
		\]
		This shows that $M\cap N\cap\omega_3\subseteq X\cap M$, and since $X\cap M\subseteq M$, this is enough to get
		\[
		M\cap N\cap\omega_3=(X\cap M)\cap N\cap\omega_3\in N.
		\]
	\end{proof}
	
	\begin{definition}
		Let $\mathcal{A}$ be a finite subset of $\mathcal{S}$, let $\nu<\omega_1$, and let $\alpha,\beta\in\omega_3$. We define the order $<_{\mathcal{A},\nu}$ by letting $\alpha<_{\mathcal{A},\nu}\beta$ if and only if $\alpha<\beta$ and there are $M_0,\dots,M_n\in\mathcal{A}$ and $\gamma_0<\dots<\gamma_{n-1}<\omega_3$ such that
		\begin{enumerate}[label=(\alph*)]
			\item $\sup_{i\leq n}\delta_{M_i}\leq\nu$,
			
			\item $\alpha\in M_0$ and $\beta\in M_n$, and
			
			\item $\gamma_i\in M_{i}\cap M_{i+1}\cap(\alpha,\beta)$, for every $i<n$.
		\end{enumerate}     
	\end{definition}
	
	\begin{remark}
		The order $<_{\mathcal{A},\nu}$ is transitive.
	\end{remark}
	
	We will denote the set $\{M\in\mathcal{A}:\delta_{M}\leq\nu\}$ by $\mathcal{A}^\nu$, and if $\alpha$ and $\beta$ are as in the definition above, we will say that they are \textit{$\mathcal{A}^\nu$-connected through $\{M_0,\dots,M_n\}$.} Moreover, if $\beta\in\field(<_{\mathcal{A},\nu})$, 
	we let $\pred_{\mathcal{A},\nu}(\beta)$ be the set of all immediate $<_{\mathcal{A},\nu}$-predecessors of $\beta$ in $\mathcal{A}^\nu$.\footnote{The field of a relation $<$ is $\field(<)=\dom(<)\cup\range(<)$, i.e., the set of all $a$ such that there is some $b$ with $b<a$ or $a<b$.} That is, $\pred_{\mathcal{A},\nu}(\beta)$ is the set of all $\alpha\in\field(<_{\mathcal{A},\nu})$ such that $\alpha<_{\mathcal{A},\nu}\beta$ and for which there is no $\gamma\in\field(<_{\mathcal{A},\nu})$ such that $\alpha<_{\mathcal{A},\nu}\gamma<_{\mathcal{A},\nu}\beta$.
	
	
	
	The forcing $\mathbb{P}$ witnessing Theorem \ref{mainthm} is defined as follows. Conditions in $\mathbb{P}$ are tuples $p=(\mathcal{M}_p, \mathcal{A}_p, a_p, d_p, u_p, b_p)$ such that:
	\begin{enumerate}[label=(C\arabic*),leftmargin=3.5em]
		\item $\mathcal{M}_p$ is an $(\mathcal{S}, \mathcal{L})$-symmetric system. 
		\item $\mathcal{A}_p\subseteq\mathcal{M}_p\cap\mathcal{S}$ is such that $X\cap M\in\mathcal{A}_p$ for all $M\in\mathcal{A}_p$ and $X\in\mathcal{M}_p\cap\mathcal{L}\cap M$. 
		\item $a_p\in[\omega_3]^{<\omega}$.
		\item $d_p\in[\omega_1]^{<\omega}$.
		\item $u_p=(u_p^\alpha\,:\,\alpha\in a_p)$ and for each $\alpha\in a_p$, $u_p^\alpha:d_p\to 2$ is a function.
		\item $b_p=(b_p(\alpha, \beta)\,:\,\alpha, \beta\in a_p, \alpha<\beta)$ and for all $\alpha<\beta$ in $a_p$,
		\begin{enumerate}
		\item $b_p(\alpha, \beta)\in [d_p]^{{<}\omega}$ and
		\item $b_p(\alpha, \beta)\sub\delta_M$ for every $M\in\mathcal{A}_p$ such that $\alpha$, $\beta\in M$.
		\end{enumerate}
		\item For every $\nu\in d_p$ and all $\alpha,\beta\in a_p$, if $\alpha<_{\mathcal{A}_p,\nu}\beta$, then $u_p^\alpha(\nu)\leq u_p^\beta(\nu)$.
		\item For all $\alpha<\beta\in a_p$ and for every $\nu\in d_p\setminus b_p(\alpha, \beta)$, $u_p^\alpha(\nu)\leq u_p^\beta(\nu)$.
	\end{enumerate}
	
	Given $\mathbb{P}$-conditions $p$ and $q$, $q$ extends $p$ (which we also denote by $q\leq p$) if and only if
	\begin{enumerate}[label=(O\arabic*),leftmargin=3.5em]
		\item $\mathcal{M}_q\supseteq\mathcal{M}_p$;
		\item $\mathcal{A}_q\supseteq\mathcal{A}_p$;
		\item $a_q\supseteq a_p$;
		\item $d_q\supseteq d_p$;
		\item for all $\alpha\in a_p$, $u_q^\alpha\supseteq u_p^\alpha$;
		\item for all $\alpha<\beta$ in $a_p$, $b_q(\alpha, \beta)=b_p(\alpha, \beta)$.
	\end{enumerate}
	
	If $p$ is a condition in $\mathbb{P}$, we will write $<_{p,\nu}$ instead of $<_{\mathcal{A}_p,\nu}$, and $\pred_{p,\nu}(\alpha)$ instead of $\pred_{\mathcal{A}_p,\nu}(\alpha)$, for every $\alpha\in a_p$. 
	
	Let us give some intuition about the definition of the forcing $\mathbb{P}$. Recall that our goal is to force a strong chain $\langle X_\alpha:\alpha<\omega_3\rangle$ of subsets of $\omega_1$. Suppose that $p$ is a condition in $\mathbb{P}$. The small models $M$ in the distinguished set $\mathcal{A}_p$, which we will call the set of \emph{active models} of $p$, are exactly the models for which we will show that $p$ is $(M,\mathbb{P})$-generic. The fact that the set of active models is not closed under isomorphism, unlike the $(\mathcal{S},\mathcal{L})$-symmetric system $\mathcal{M}_p$, is key in making sure that 
	 the arguments in many of the proofs that come next go through. For every $\alpha\in a_p$, the function $u_p^\alpha:d_p\to 2$ is a finite approximation of the characteristic function of the $\alpha$-th set $X_\alpha$ of the strong chain that we want to add generically. Specifically, $d_p$ gives us ordinals $\nu<\omega_1$, and $u_p^\alpha$ decides whether $\nu$ will be an element of $X_\alpha$ (i.e., $u_p^\alpha(\nu)=1$), or it will not (i.e., $u_p^\alpha(\nu)=0$). The component $b_p$ introduces, in clause (C8), the promise that the disagreement between the sets $X_\alpha$ and $X_\beta$ will be at most the finite set $b_p(\alpha,\beta)$.

	Clause (C7) is the main ingredient in the definition. Let $\alpha,\beta\in a_p$ be such that $\alpha<\beta$. Suppose that $\alpha,\beta\in M$ for some $M\in\mathcal{A}_p$. We want to promise that $u_p^\alpha(\nu)\leq u_p^\beta(\nu)$ for all $\nu\in d_p$ such that $\delta_M\leq\nu$. In other words, if $u_p^\alpha$ places $\nu$ in $X_\alpha$, then $u_p^\beta$ must place $\nu$ in $X_\beta$ as well, for all $\nu<\omega_1$, except maybe those $\nu$ that belong to $M$. However, the situation is a little more complicated. Let us momentarily denote by \emph{naive-(C7)} this weak version of clause (C7). We need to anticipate some potential issues that may appear when amalgamating two different conditions $p,q\in\mathbb{P}$. Let $\nu\in d_p$ and suppose that $\alpha,\beta\in a_p$ are such that $\alpha<\beta$. Moreover, suppose that $\alpha\in M_0$ and $\beta\in M_1$ for some $M_0,M_1\in\mathcal{A}_p^\nu$, but there is no $M\in\mathcal{A}_p^\nu$ such that $\alpha,\beta\in M$. In this case, clause naive-(C7) does not impose any requirement on the values of $u_p^\alpha(\nu)$ and $u_p^\beta(\nu)$. So it might very well happen that $u_p^\alpha(\nu)>u_p^\beta(\nu)$. Suppose now that $\gamma\in a_q$ such that $\alpha<\gamma<\beta$ and $\gamma\in M_0\cap M_1$. Note that in this case, it would be impossible to amalgamate $p$ and $q$. Indeed, suppose that $r\in\mathbb{P}$ were a condition extending both $p$ and $q$. Then, on the one hand, $u_r^\alpha(\nu)=u_p^\alpha(\nu)>u_p^\beta(\nu)=u_r^\beta(\nu)$. However, on the other hand, naive-(C7) would impose $u_r^\alpha(\nu)\leq u_r^\gamma(\nu)=u_q^\gamma(\nu)\leq u_r^\beta(\nu)$, which is impossible. The order $<_{\mathcal{A},\nu}$ is designed precisely to account for these potential issues that might appear in some of the proofs of the amalgamation lemmas, and clause (C7) is defined accordingly.

	
	

	
	
	\subsection{Density lemmas and basic properties of the forcing}\label{subsect-sc-density-lemmas}
	
	\begin{lemma}\label{dens0}
		For every $p\in\mathbb{P}$ and every $\alpha\in\omega_3$, there are $p^*\in\mathbb{P}$ extending $p$ and  $\alpha^*<\omega_3$ such that $\alpha^*>\alpha$ and $\alpha^*\in a_{p^*}$. 
	\end{lemma}
	
	\begin{proof}
		Let $\alpha^*<\omega_3$ such that $\alpha^*>\alpha$ and  $\alpha^*>\sup(M\cap \omega_3)$ for each $M\in\mathcal{A}_p$.  Hence, $\alpha^*$ is not $\mathcal{A}_p^\nu$-connected to any $\beta\in\mathcal{A}_p$, for any $\nu\in d_p$. We may also assume, for notational convenience in the definition below of $b^*$, that $\alpha^*>\beta$ for all $\beta\in a_p$. 
		
		For every $\nu\in d_p$, we let $u_{p^*}^{\alpha^*}(\nu)=0$ and $u_{p^*}^\beta(\nu)=u_p^\beta(\nu)$ for each $\beta\in a_p$. Let also $b^*=(b^*(\alpha, \beta)\,:\,\alpha, \beta\in a_p\cup\{\alpha^*\}, \alpha <\beta)$, where $b^*(\alpha, \beta)=b_p(\alpha, \beta)$ if $\beta<\alpha^*$ and $b^*(\alpha, \alpha^*)=d_p$ for all $\alpha\in a_p$. Then,
		\[
		p^*=(\mathcal{M}_p,\mathcal{A}_p,a_p\cup\{\alpha^*\},d_p,(u_{p^*}^\beta:\beta\in a_p\cup\{\alpha^*\}), b^*)
		\]
		is a condition in $\mathbb{P}$ extending $p$ as desired.
	\end{proof}
	
	
	
	
	The proof of the next density lemma is immediate.
	
	\begin{lemma}\label{dens1}
		For every $p\in\mathbb{P}$ and every $\nu\in\omega_1\setminus d_p$, there exists some $p^*\in\mathbb{P}$ such that $p^*\leq p$, $\nu\in d_{p^*}$, and for any $i\in \{0,1\}$, $u_{p^*}^\alpha(\nu)=i$ for all $\alpha\in a_{p^*}$.
	\end{lemma}

	
	\begin{lemma}\label{dens2}
		Let $p\in\mathbb{P}$, $\alpha\in a_p$, and $\nu<\omega_1$. For each $\beta\in a_p$ such that $\beta>\alpha$ there are $p^*\in\mathbb{P}$ such that $p^*\leq p$ and $\nu^*<\omega_1$ above $\nu$ such that $\nu^*\in d_{p^*}$, $u_{p^*}^\alpha(\nu^*)=0$, and $u_{p^*}^\beta(\nu^*)=1$.
	\end{lemma}
	
	\begin{proof}
		Let $\nu^*\in\omega_1\setminus d_p$ be larger than both $\nu$ and $\delta_M$, for every $M\in\mathcal{A}_p$. Let us define $d_{p^*}=d_p\cup\{\nu^*\}$. For every $\gamma\in a_p$, we extend $u_p^\gamma$ to a function $u_{p^*}^\gamma$ with domain $d_{p^*}$ such that $u_{p^*}^\gamma(\nu^*)=0$ if $\gamma<\beta$, and $u_{p^*}^\gamma(\nu^*)=1$ if $\gamma\geq\beta$. Then,
		\[
		p^*=(\mathcal{M}_p,\mathcal{A}_p,a_p,d_{p^*},(u_{p^*}^\gamma:\gamma\in a_p), b_p)
		\]
		is a condition in $\mathbb{P}$ extending $p$ as desired.
	\end{proof}
	
	
	\begin{lemma}\label{ontop}
		Let $Q\in\mathcal{S}\cup\mathcal{L}$ and let $p\in\mathbb{P}\cap Q$. Then there is a condition $q\in\mathbb{P}$ extending $p$ and such that 
		\begin{enumerate}[label=(\arabic*)]
			\item $Q\in\mathcal{M}_q$ and
			\item $Q\in\mathcal{A}_q$ if $Q$ is countable.
		\end{enumerate}
	\end{lemma}
	\begin{proof}
		Suppose first that $Q\in\mathcal{S}$ and let $q=(\mathcal{M}_q,\mathcal{A}_q, a_p, d_p, u_p, b_p)$ be such that 
		\[
		\mathcal{M}_q=\mathcal{M}_p\cup\{Q\}\cup\{X\cap Q\,:\,X\in\mathcal{M}_p\cap\mathcal{L}\}
		\]
		and 
		\[
		\mathcal{A}_q=\mathcal{A}_p\cup\{Q\}\cup\{X\cap Q\,:\,X\in\mathcal{M}_p\cap\mathcal{L}\}.
		\]
		The only clause in the definition of $\mathbb{P}$ that we need to check is (C7), as clause (C1) follows from Lemma \ref{amal-ontop}, and the other ones are obvious. On the one hand, note that since $p\in Q$, in particular $d_p\subseteq Q$, and hence $\nu<\delta_Q$ for all $\nu\in d_p$. On the other hand, note that $\delta_Q=\delta_{X\cap Q}$ for all $X\in\mathcal{M}_p\cap\mathcal{L}$. Therefore, neither $Q$ nor the models $X\cap Q$, where $X\in\mathcal{M}_p\cap\mathcal{L}$, add new $\mathcal{A}_q^\nu$-connections between the elements of $a_p$, for any $\nu\in d_p$. Hence, $q$ satisfies clause (C7).
		
		The proof for the uncountable case is exactly the same as for the countable case, with the only change in the definition of $q$ being that $\mathcal{M}_q=\mathcal{M}_p\cup\{Q\}$ and $\mathcal{A}_q=\mathcal{A}_p$. 
	\end{proof}
	
	\begin{definition}
		Given a condition $p\in\mathbb{P}$ and a model $Q\in\mathcal{M}_p$, we define $p\upharpoonright Q=(\mathcal{M}_{p\upharpoonright Q}, \mathcal{A}_{p\upharpoonright Q}, a_{p\upharpoonright Q}, d_{p\upharpoonright Q}, u_{p\upharpoonright Q}, b_{p\upharpoonright Q})$ by letting
		\begin{enumerate}[label=(\arabic*)]
			\item $\mathcal{M}_{p\upharpoonright Q}=\mathcal{M}_p\cap Q$,
			
			\item $\mathcal{A}_{p\upharpoonright Q}=\mathcal{A}_p\cap Q$,
			
			\item $a_{p\upharpoonright Q}=a_p\cap Q$,
			
			\item $d_{p\upharpoonright Q}=d_p\cap Q$, 	
			
			\item $u_{p\upharpoonright Q}^\alpha(\nu)=u_p^\alpha(\nu)$, for all $\alpha\in a_{p\upharpoonright Q}$ and all $\nu\in d_{p\upharpoonright Q}$, and
			\item $b_{p\upharpoonright Q}=(b_p(\alpha, \beta)\,:\,\alpha, \beta\in a_p\cap Q, \alpha<\beta)$.
		\end{enumerate}
		and we call it \emph{the restriction of $p$ to $Q$} .
	\end{definition}
	
	\begin{lemma}\label{restriction}
		If $p\in\mathbb{P}$ and $Q\in\mathcal{M}_p$, then $p\upharpoonright Q$ is a condition in $\mathbb{P}\cap Q$ such that $p\leq p\upharpoonright Q$ if
		\begin{enumerate}[label=(\arabic*)]
			\item $Q\in\mathcal{L}$ or
			\item $Q\in\mathcal{S}$ and $Q\in\mathcal{A}_p$.
		\end{enumerate}
	\end{lemma}
	
	\begin{proof}
		We only prove the case where $Q$ is countable and $Q\in\mathcal{A}_p$, as the uncountable case is straightforward. It is easily seen that $p\upharpoonright Q$ is a condition in $\mathbb{P}$. Indeed, clause (C1) in the definition of $\mathbb{P}$ follows from Lemma \ref{amal-rest} and the other clauses are immediate. Since $p\upharpoonright Q$ is finite, it clearly belongs to $Q$. Lastly, $p\leq p\upharpoonright Q$ follows easily from the definition of $p\upharpoonright Q$.
	\end{proof}
	
	\subsection{Properness}\label{subsect-sc-S-properness}
	
	\begin{definition}\label{defreflectionS}
		Let $p\in\mathbb{P}$ and $M\in\mathcal{A}_p$. A condition $q\in\mathbb{P}$ is called an \emph{$(\mathcal{S},M)$-reflection of $p$} if it satisfies the following properties:
		\begin{enumerate}[label=(\arabic*)]
			\item $q\leq p\upharpoonright M$.
			
			\item If $\mathcal{A}_p=\{M_0,\dots, M_n\}$, then there are $M_0',\dots M_n'\in\mathcal{S}$ such that $\mathcal{A}_q=\{M_0',\dots,M_n'\}$ with the following properties: 
			\begin{enumerate}
				\item[(2.a)] $M_i'=M_i$, for all $i\leq n$ such that $M_i\in M$.
				
				\item[(2.b)] $M_i'\cap M\cap\omega_3=M_i\cap M\cap\omega_3$, for all $i\leq n$ such that $\delta_{M_i}<\delta_M$. 
				
				\item[(2.c)] $a_q\cap M_i\cap M=a_p\cap M_i\cap M$, for all $i\leq n$ such that $\delta_{M_i}<\delta_M$.
				
				\item[(2.d)] $\delta_{M_i'}=\delta_{M_i}$, for all $i\leq n$ such that $\delta_{M_i}<\delta_M$. 
				
				\item[(2.e)] If $\alpha<\beta$ are in $M\cap\omega_3$, $\nu\in M\cap\omega_1$, and there is $A\subseteq \mathcal{A}_p^\nu$ such that $\alpha$ and $\beta$ are $\mathcal{A}_p^\nu$-connected through $A$, then $\alpha$ and $\beta$ are $\mathcal{A}_q^\nu$-connected through $A'=\{M_i':M_i\in A\}$.
				
				\item[(2.f)] $\delta_{M_i'}=\delta_{M_i}$ and $M_i'\cap M\cap\omega_3=M_i\cap M\cap\omega_3$, for every $i\leq n$ and every $M_i'\in\mathcal{A}_q^\nu$ such that $\nu\in d_p\cap M$.
			\end{enumerate}
		\end{enumerate}
	\end{definition}
	
	The following result is straightforward.
	
	\begin{lemma}\label{selfreflectionS}
		Let $p\in\mathbb{P}$ and $M\in\mathcal{A}_p$. Then $p$ is an $(\mathcal{S},M)$-reflection of itself.
	\end{lemma}


	
	\begin{lemma}\label{amalgamation1S}
		Let $p\in\mathbb{P}$ and let $M\in\mathcal{A}_p$. Let $q\in\mathbb{P}\cap M$ be an $(\mathcal{S},M)$-reflection of $p$. Then there is a condition $r\in\mathbb{P}$ extending both $p$ and $q$.
	\end{lemma}

	\begin{proof}
		Let $\mathcal{M}_r$ be the $(\mathcal{S},\mathcal{L})$-symmetric system extending both $\mathcal{M}_p$ and $\mathcal{M}_q$ given by Lemma \ref{lemma-amal}. Let  
		\[
		\mathcal{A}_r:=\mathcal{A}_p\cup\mathcal{A}_q\cup \{X\cap N\,:\, N\in \mathcal{A}_p\cup\mathcal{A}_q,\,X\in\mathcal{M}_r\cap\mathcal{L}\cap N\},
		\]
		and let $a_r:=a_p\cup a_q$ and $d_r:=d_p\cup d_q$. Let also $b_r$ be the extension of $b_p\cup b_q$ to $\{(\alpha, \beta)\in a_r\times a_r, \alpha<\beta\}$ defined by letting $b_r(\alpha, \beta)$ be, for all $\alpha<\beta$ in $a_r$ with $\alpha\in a_p$ iff $\beta\notin a_p$, equal to $d_r\cap\delta_{M'}$ for any $M'\in\mathcal{A}_r$ such that $\alpha$, $\beta\in M'$ of minimal $\omega_1$-height (and letting $b_r(\alpha, \beta)=d_r$ if there is no such $M'$).
		
		Fix some $\nu\in d_r$ and suppose that $a_r=\{\alpha_0,\dots,\alpha_n\}$, where $\alpha_i<\alpha_{i+1}$ for each $i<n$. We will define $u_r^{\alpha_i}(\nu)$ by induction on $i\leq n$, while making sure that clause (C7) of the definition of the poset $\mathbb{P}$ is satisfied.
		\vspace{0.3cm}
		
		\textbf{Case 1.} Suppose first that $\nu\in d_{p\upharpoonright M}$. For every $i\leq n$, define $u_r^{\alpha_i}(\nu)$ as 
		
		\[
		u_r^{\alpha_i}(\nu)=
		\begin{cases}
			u_q^{\alpha_i}(\nu), &\text{if }\alpha_i\in a_q\\
			u_p^{\alpha_i}(\nu), &\text{if }\alpha_i\in a_p\setminus M.
		\end{cases}
		\]
		Let us check that condition (C7) is satisfied. So, let $i<i^*\leq n$ such that $\alpha_i<_{r,\nu}\alpha_{i^*}$. We want to prove that $u_r^{\alpha_i}(\nu)\leq u_r^{\alpha_{i^*}}(\nu)$. We will show that, necessarily, $\alpha_i,\alpha_{i^*}\in a_p$ and $\alpha_i<_{p,\nu}\alpha_{i^*}$.
		
		Let $M_0,\dots, M_m\in\mathcal{A}_r^\nu$ and $\gamma_0<\dots<\gamma_{m-1}<\omega_3$ in $(\alpha_i,\alpha_{i^*})$ such that $\alpha_i\in M_0$, $\alpha_{i^*}\in M_m$, and $\gamma_j\in M_j\cap M_{j+1}$ for every $j<m$. First, note that if $M_j$ is of the form $X\cap N_j$, where $N_j\in\mathcal{A}_p^\nu\cup\mathcal{A}_q^\nu$ and $X\in\mathcal{M}_r\cap\mathcal{L}\cap N_j$, then $\delta_{M_j}=\delta_{N_j}$ and $\gamma_j,\gamma_{j+1}\in M_j\subseteq N_j$. Hence, $\alpha_i$ and $\alpha_{i^*}$ remain $\mathcal{A}_r^\nu$-connected if we substitute all the models $M_j$ as above by $N_j$. Thus, we may assume that $M_j\in\mathcal{A}_p^\nu\cup\mathcal{A}_q^\nu$ for all $j\leq m$. Let us also remark that $\delta_{M_j}\leq\nu<\delta_M$ for every $j\leq m$.
		
		Let us first show that $\alpha_i,\alpha_{i^*}\in a_p$. Let $\alpha\in\{\alpha_i,\alpha_{i^*}\}$ and $N\in\{M_0,M_m\}$ be such that $\alpha\in N$, and assume that $\alpha\in a_q$. If $N\in\mathcal{A}_p^\nu$, since $\delta_N\leq\nu<\delta_M$ and $q$ is an $(\mathcal{S},M)$-reflection of $p$, by (2.c) of Definition \ref{defreflectionS}, we have
		\[
		\alpha\in a_q\cap N=a_p\cap N\cap M.
		\]
		If $N\in\mathcal{A}_q^\nu$, since $\alpha\in a_q\subseteq M$ and $\delta_N\leq\nu\in d_p\cap M$, by (2.f) of Definition \ref{defreflectionS}, there is some $N^*\in\mathcal{A}_p^\nu$ such that $\delta_{N^*}=\delta_N$ and 
		\[
		\alpha\in N\cap\omega_3=N^*\cap M\cap\omega_3.
		\]
		Hence, by (2.c), we have $\alpha\in a_q\cap N^*=a_p\cap N^*\cap M$. Since $a_r=a_p\cup a_q$, this finishes the proof of $\alpha_i,\alpha_{i^*}\in a_p$.
		
		Let us now show that $\alpha_i<_{p,\nu}\alpha_{i^*}$. Note that for every $j<m$ such that $M_{j+1}\in\mathcal{A}_q^\nu$, since $\delta_{M_{j+1}}\leq\nu\in d_p\cap M$, by (2.f) of Definition \ref{defreflectionS} there is some $M_{j+1}^*\in\mathcal{A}_p^\nu$ such that $\delta_{M_{j+1}^*}=\delta_{M_{j+1}}$ and
		\[
		\gamma_j,\gamma_{j+1}\in M_{j+1}\cap\omega_3=M_{j+1}^*\cap M\cap\omega_3.
		\]
		This shows that $\alpha_i<_{p,\nu}\alpha_{i^*}$.
		
		Therefore, since $\alpha_i,\alpha_{i^*}\in a_p$ and $\alpha_i<_{p,\nu}\alpha_{i^*}$, we can conclude that
		\[
		u_r^{\alpha_i}(\nu)=u_p^{\alpha_i}(\nu)\leq	u_p^{\alpha_{i^*}}(\nu)=	u_r^{\alpha_{i^*}}(\nu).
		\]
		
		\textbf{Case 2.} Suppose now that $\nu\in d_q\setminus d_p$. Since $r$ has to extend $q$, we will simply let $u_r^{\alpha_i}(\nu)=u_q^{\alpha_i}(\nu)$, for all $\alpha_i\in a_q$. Hence, we only need to define $u_r^{\alpha_i}(\nu)$ for $\alpha_i\in a_p\setminus M$. We argue by induction on $i\leq n$.
		
		Assume first that $i=0$. If $\alpha_0\in a_q$, we have already defined $u_r^{\alpha_0}(\nu)$ as $u_r^{\alpha_0}(\nu)=u_q^{\alpha_0}(\nu)$. If $\alpha_0\in a_p\setminus M$, let $u_r^{\alpha_0}(\nu)=0$.
		
		Now, let $k\leq n$ and assume that we have defined $u_r^{\alpha_i}(\nu)$ for every $i<k$, and that clause (C7) of the definition of $\mathbb{P}$ holds for every pair $\alpha_{i_0},\alpha_{i_1}\in a_r$, where $i_0<i_1<k$. That is, if $\alpha_{i_0}<_{r,\nu}\alpha_{i_1}$, then $u_r^{\alpha_{i_0}}(\nu)\leq u_r^{\alpha_{i_1}}(\nu)$. Moreover, if $i<k$ and $\alpha_{i}\in a_p\setminus M$, we assume that $u_r^{\alpha_{i}}(\nu)=0$, unless there is some $<_{r,\nu}$-predecessor $\alpha_j$ of $\alpha_i$ in $a_r$ such that $u_r^{\alpha_j}(\nu)=1$. Then, of course, $u_r^{\alpha_{i}}(\nu)=1$.  Note that, in particular, by the induction hypothesis, there must be some $<_{r,\nu}$-predecessor $\alpha_j$ of $\alpha_i$ in $a_q$ such that $u_r^{\alpha_j}(\nu)=u_q^{\alpha_j}(\nu)=1$.
		
		
		\textbf{Case 2.a.} If $\alpha_k$ does not have any $<_{r,\nu}$-predecessor in $a_r$, then we define $u_r^{\alpha_k}(\nu)$ exactly as we did in the case $i=0$. That is, 
		\[
		u_r^{\alpha_k}(\nu)=
		\begin{cases}
			u_q^{\alpha_k}(\nu), &\text{if }\alpha_k\in a_q\\
			0, &\text{if }\alpha_k\in a_p\setminus M.
		\end{cases}
		\]
		In this case condition (C7) of the definition of the poset $\mathbb{P}$ is vacuously satisfied.
		
		\textbf{Case 2.b.} If $\alpha_k$ has $<_{r,\nu}$-predecessors in $a_r$, we define $a_r^{\alpha_k}(\nu)$ as follows:
		\[
		u_r^{\alpha_k}(\nu)=
		\begin{cases}
			u_q^{\alpha_k}(\nu), &\text{if }\alpha_k\in a_q\\
			1, &\text{if }\alpha_k\in a_p\setminus M\text{ and }\exists\alpha_j\in\pred_{r,\nu}(\alpha_k)(u_r^{\alpha_j}(\nu)=1)\\
			0, &\text{if }\alpha_k\in a_p\setminus M\text{ and }\forall\alpha_j\in\pred_{r,\nu}(\alpha_k)(u_r^{\alpha_j}(\nu)=0).
		\end{cases}
		\]
		If $\alpha_k\in a_p\setminus M$, then clause (C7) holds simply by the definition of $u_r^{\alpha_k}(\nu)$. However, if $\alpha_k\in a_q$, we need to prove that for every $i<k$, if $\alpha_i<_{r,\nu}\alpha_k$, then $u_r^{\alpha_i}(\nu)\leq u_r^{\alpha_k}(\nu)$. Assume that $\alpha_k\in a_q$. Note that if $u_r^{\alpha_i}(\nu)=0$ for every $i<k$ such that $\alpha_i<_{r,\nu}\alpha_k$, then we trivially get $u_r^{\alpha_i}(\nu)\leq u_r^{\alpha_k}(\nu)$. Hence, we may also assume that there is some $i<k$ such that $\alpha_i<_{r,\nu}\alpha_k$ and $u_r^{\alpha_i}(\nu)=1$, and thus we need to show that $u_r^{\alpha_k}(\nu)=u_q^{\alpha_k}(\nu)=1$. Note that if $\alpha_i\in a_p\setminus M$, then by the induction hypothesis there must be some $j<i$ such that $\alpha_j\in a_q$, $\alpha_j<_{r,\nu}\alpha_i$ and $u_r^{\alpha_j}(\nu)=u_q^{\alpha_j}(\nu)=1$, and by the transitivity of $<_{r,\nu}$, we have $\alpha_j<_{r,\nu}\alpha_k$. Therefore, the following claim is enough to ensure that $u_r^{\alpha_k}(\nu)=u_q^{\alpha_k}(\nu)=1$.
		
		\begin{claim}
			For every $j^*<k$ such that $\alpha_{j^*}\in a_q$, if $\alpha_{j^*}<_{r,\nu}\alpha_k$, then $\alpha_{j^*}<_{q,\nu}\alpha_k$.
		\end{claim}
		\begin{proof}
			Let $M_0,\dots, M_m\in\mathcal{A}_r^\nu$ and $\gamma_0<\dots<\gamma_{m-1}<\omega_3$ in $(\alpha_{j^*},\alpha_k)$ such that $\alpha_{j^*}\in M_0$, $\alpha_k\in M_m$, and $\gamma_i\in M_i\cap M_{i+1}$ for each $i<m$. By the same argument as in Case 1, we may assume that $M_i\in\mathcal{A}_p^\nu\cup\mathcal{A}_q^\nu$ for all $i\leq m$. It is also worth noting that $\delta_{M_i}\leq\nu<\delta_M$ for every $i\leq m$, since $M_i\in\mathcal{A}_r^\nu$ and $\nu\in d_q\subseteq M$. We will show that there are $M_0^*,\dots,M_m^*\in\mathcal{A}_q^\nu$ such that $\alpha_{j^*}$ and $\alpha_k$ are $\mathcal{A}_q^\nu$-connected through $\{M_0^*,\dots,M_m^*\}$. 
			
			If there is no $i\leq m$ such that $M_i\in\mathcal{A}_q^\nu$, then $\alpha_{j^*}<_{\mathcal{A}_p,\nu}\alpha_k$. Hence, since $q$ is an $(\mathcal{S},M)$-reflection of $p$, and $\alpha_{j^*},\alpha_k\in a_q\subseteq M$ and $\nu\in d_q\subseteq M$, by item (2.e) of Definition \ref{defreflectionS}, we have that $\alpha_{j^*}<_{q,\nu}\alpha_k$ and we are done. Otherwise, let $i_0$ be the least $i\leq m$ for which $M_i\in\mathcal{A}_q^\nu$. If $i_0=0$, we let $M_0^*:=M_0$. If $i_0>0$, we have that $\gamma_{i_0}\in M_{i_0}\subseteq M$, and thus, by item (2.e) of the definition of $(\mathcal{S},M)$-reflection there are $M_0',\dots,M_{i_0-1}'\in\mathcal{A}_q^\nu$ such that $\alpha_{j^*}$ and $\gamma_{i_0}$ are $\mathcal{A}_q^\nu$-connected through $\{M_0',\dots, M_{i_0-1}'\}$. Let $M_i^*:=M_i'$ for every $i<i_0$, and $M_{i_0}^*:=M_{i_0}$. Now, if there is no $i\leq m$ such that $i_0<i$ and $M_i\in\mathcal{A}_q^\nu$, we can apply item (2.e) of Definition \ref{defreflectionS} again to show that $\gamma_{i_0+1}$ and $\alpha_k$ are $\mathcal{A}_q^\nu$-connected. Otherwise, we let $i_1$ be the least $i\leq m$ such that $i>i_0$ and $M_i\in\mathcal{A}_q^\nu$. By the same argument as above, using item (2.e) of the definition of $(\mathcal{S},M)$-reflection, we can prove that $\gamma_{i_0+1}$ and $\gamma_{i_1}$ are $\mathcal{A}_q^\nu$-connected. At this point it should be clear that we can repeat this argument, which will finish eventually because $\alpha_k\in M\cap\omega_3$, to get $M_0^*,\dots,M_m^*\in\mathcal{A}_q^\nu$ such that $\alpha_{j^*}$ and $\alpha_k$ are $\mathcal{A}_q^\nu$-connected through $\{M_0^*,\dots,M_m^*\}$, as we wanted.
		\end{proof}
		
		
		\textbf{Case 3.} Lastly, suppose that $\nu\in d_p\setminus d_q=d_p\setminus M$. Since $r$ has to extend $p$, we will simply let $u_r^{\alpha_i}(\nu)=u_p^{\alpha_i}(\nu)$, for all $\alpha_i\in a_p$. Hence, we only need to define $u_r^{\alpha_i}(\nu)$ for $\alpha_i\in a_q\setminus a_p$. We argue by induction on $i\leq n$.
		
		Assume first that $i=0$. If $\alpha_0\in a_p$, we have already defined $u_r^{\alpha_0}(\nu)$ as $u_r^{\alpha_0}(\nu)=u_p^{\alpha_0}(\nu)$. If $\alpha_0\in a_q\setminus a_p$, we let $u_r^{\alpha_0}(\nu)=0$. 
		
		As in the last case, let $k\leq n$, and assume that we have defined $u_r^{\alpha_i}(\nu)$ for every $i<k$, and that $u_r^{\alpha_{i_0}}(\nu)\leq u_r^{\alpha_{i_1}}(\nu)$ for every pair $i_0<i_1<k$ such that $\alpha_{i_0}<_{r,\nu}\alpha_{i_1}$. Moreover, if $i<k$ and $\alpha_i\in a_q\setminus a_p$, assume that $u_{r}^{\alpha_i}(\nu)=0$, unless there is some $<_{r,\nu}$-predecessor $\alpha_j$ of $\alpha_i$ in $a_r$ such that $u_r^{\alpha_j}(\nu)=1$. Then, of course, $u_r^{\alpha_i}(\nu)=1$. Note that, in particular, by the induction hypothesis there must be some $<_{r,\nu}$-predecessor $\alpha_j$ of $\alpha_i$ in $a_p$ such that $u_r^{\alpha_j}(\nu)=u_p^{\alpha_j}(\nu)=1$.
		
		
		\textbf{Case 3.a.} If $\alpha_k$ does not have any $<_{r,\nu}$-predecessor in $a_r$, then we define $u_r^{\alpha_k}(\nu)$ as follows: 
		\[
		u_r^{\alpha_k}(\nu)=
		\begin{cases}
			u_p^{\alpha_k}(\nu), &\text{if }\alpha_k\in a_p\\
			0, &\text{if }\alpha_k\in a_q\setminus a_p.
		\end{cases}
		\]
		In this case, condition (C7) of the definition of the poset $\mathbb{P}$ is vacuously satisfied.
		
		\textbf{Case 3.b.} If $\alpha_k$ has $<_{r,\nu}$-predecessors in $a_r$, we define $a_r^{\alpha_k}(\nu)$ as follows:
		\[
		u_r^{\alpha_k}(\nu)=
		\begin{cases}
			u_p^{\alpha_k}(\nu), &\text{if }\alpha_k\in a_p\\
			1, &\text{if }\alpha_k\in a_q\setminus q_p\text{ and }\exists\alpha_j\in\pred_{r,\nu}(\alpha_k)(u_r^{\alpha_j}(\nu)=1)\\
			0, &\text{if }\alpha_k\in a_q\setminus a_p\text{ and }\forall\alpha_j\in\pred_{r,\nu}(\alpha_k)(u_r^{\alpha_j}(\nu)=0).
		\end{cases}
		\]
		If $\alpha_k\in a_q\setminus a_p$, then clause (C7) holds simply by the definition of $u_r^{\alpha_k}(\nu)$. However, if $\alpha_k\in a_p$, we need to prove that for every $i<k$, if $\alpha_i<_{r,\nu}\alpha_k$, then $u_r^{\alpha_i}(\nu)\leq u_r^{\alpha_k}(\nu)$. Assume that $\alpha_k\in a_p$. Note that if $u_r^{\alpha_i}(\nu)=0$ for every $i<k$ such that $\alpha_i<_{r,\nu}\alpha_k$, then we trivially get $u_r^{\alpha_i}(\nu)\leq u_r^{\alpha_k}(\nu)$. Hence, we may also assume that there is some $i<k$ such that $\alpha_i<_{r,\nu}\alpha_k$ and $u_r^{\alpha_i}(\nu)=1$, and thus we need to show that $u_r^{\alpha_k}(\nu)=u_p^{\alpha_k}(\nu)=1$. Note that if $\alpha_i\in a_q\setminus a_p$, then by the induction hypothesis there must be some $j<i$ such that $\alpha_j\in a_p$, $\alpha_j<_{r,\nu}\alpha_i$ and $u_r^{\alpha_j}(\nu)=u_p^{\alpha_j}(\nu)=1$, and by the transitivity of $<_{r,\nu}$, we have $\alpha_j<_{r,\nu}\alpha_k$. Therefore, the following claim is enough to ensure that $u_r^{\alpha_k}(\nu)=u_p^{\alpha_k}(\nu)=1$.
		
		\begin{claim}
			For every $j^*<k$ such that $\alpha_{j^*}\in a_p$, if $\alpha_{j^*}<_{r,\nu}\alpha_k$, then $\alpha_{j^*}<_{p,\nu}\alpha_k$.
		\end{claim}
		\begin{proof}
			Suppose that $\alpha_{j^*}$ and $\alpha_k$ are $\mathcal{A}_r^\nu$-connected through $M_0,\dots,M_m$, for some $m<\omega$. Note that for every $i\leq m$, if $M_i\in\mathcal{A}_q^\nu$, or if $M_i$ is of the form $X\cap M_i^*$ for some $M_i^*\in\mathcal{A}_q^\nu$ and some $X\in\mathcal{M}_r\cap\mathcal{L}\cap M_i^*$, since $\mathcal{A}_q\subseteq M$, we then have $M_i\in M$, and thus $M_i\subseteq M$. Hence, since $M\in\mathcal{A}_p^\nu$ (recall that $\nu\in d_p\setminus M$ by assumption, so $\delta_M\leq\nu$), if we substitute each $M_i$ as above by $M$, then $\alpha_{j^*}$ and $\alpha_k$ are $\mathcal{A}_p^\nu$-connected. That is, $\alpha_{j^*}<_{p,\nu}\alpha_k$, as we wanted.
		\end{proof}
		
		This finishes the construction of $u_r$ and the verification that is satisfies clause (C7) from the definition of condition. We still need to show that $u_r^\alpha(\nu)\leq u_r^\beta(\nu)$ holds for all $\alpha<\beta$ in $a_r$ and for every $\nu\in d_r\setminus b_r(\alpha, \beta)$. We may of course assume that $\alpha\in a_p$ iff $\beta\notin a_p$ as otherwise we are done since $p$ and $q$ are conditions. But then, by the definition of $b_r(\alpha, \beta)$, there is some $M'\in\mathcal{A}_r$ such that $\alpha$, $\beta\in M'$ and $\delta_{M'}\leq\nu$. In particular, $\alpha<_{r, \nu}\beta$, so $u_r^\alpha(\nu)\leq u_r^\beta(\nu)$ as desired.

		We have thus shown that $$r:=(\mathcal{M}_r,\mathcal{A}_r,a_r,d_r,(u_r^\alpha:\alpha\in a_r), b_r)$$ is a condition in $\mathbb{P}$ extending both $p$ and $q$.
	\end{proof}
	
	\begin{lemma}\label{propS}
		The forcing $\mathbb{P}$ is proper for countable elementary submodels.
	\end{lemma}
	
	\begin{proof}
		Let $p\in\mathbb{P}$ and let $M^*$ be a countable elementary submodel of a large enough $H(\theta)$ containing $p$ and $\mathbb{P}$. Let $M:=M^*\cap H(\omega_3)$. By Lemma \ref{ontop}, we may find a condition $p'$ extending $p$ and such that $M\in\mathcal{A}_{p'}$. We will show that $p'$ is $(M^*,\mathbb{P})$-generic. So, let $D\in M^*$ be a dense subset of $\mathbb{P}$, and let $p^*\in D$ be such that $p^*\leq p'$. Note that, in light of Lemma \ref{amalgamation1S}, it suffices to find an $(\mathcal{S},M)$-reflection $q$ of $p^*$ such that $q\in D\cap M^*$. Observe that all the parameters in the definition of $(\mathcal{S},M)$-reflection (Definition \ref{defreflectionS}) are members of $M^*$: Item (1) has parameters in $M$ by Lemma \ref{restriction}, items (2.a), (2.d) and (2.e) are clear, and items (2.b), (2.c) and (2.f) follow from Lemma \ref{key-claim}. Therefore, since $p^*$ is an $(\mathcal{S},M)$-reflection of itself by Lemma \ref{selfreflectionS}, we can find an $(\mathcal{S},M)$-reflection of $p^*$ in $D\cap M^*$ by the elementarity of $M^*$. 
	\end{proof}

	\begin{remark}
		Lemma \ref{propS}, which crucially depends on Lemma \ref{key-claim}, which in turn depends on Lemma \ref{agreementiso}, is the only lemma in the proof of Theorem \ref{mainthm} which does not go through if we try to force the same object consisting of partial functions indexed by ordinals not in $\omega_3$ but in $\omega_4$ (or anything higher, of course). The reason is that the versions of Lemmas \ref{agreementiso} and \ref{key-claim} in this situation (i.e., replacing $\omega_3$ with $\omega_4$) do not hold. 
	\end{remark}

	\subsection{$\aleph_1$-properness}\label{subsect-sc-L-properness}
	
	\begin{definition}\label{defreflectionL}
		Let $p\in\mathbb{P}$ and $X\in\mathcal{M}_p\cap\mathcal{L}$. A condition $q\in\mathbb{P}$ is called an \emph{$(\mathcal{L},X)$-reflection of $p$} if it satisfies the following properties:
		\begin{enumerate}[label=(\arabic*)]
			\item $q\leq p\upharpoonright X$.
			\item $d_p=d_q$.
			\item If $\mathcal{A}_p=\{M_0,\dots, M_n\}$, then there are $M_0',\dots M_n'\in\mathcal{S}$ such that $\mathcal{A}_q=\{M_0',\dots,M_n'\}$, and for every $i\leq n$,
			\begin{enumerate}
				\item[(3.a)] $M_i=M_i'$ if $M_i\in\mathcal{A}_p\cap X$,
				\item[(3.b)] $X\cap M_i\cap\omega_3=X\cap M_i'\cap\omega_3$,
				\item[(3.c)] $a_q\cap M_i\cap X=a_p\cap M_i\cap X$, and
				\item[(3.d)] $\delta_{M_i'}=\delta_{M_i}$.
			\end{enumerate}
			
			
		\end{enumerate}
	\end{definition}
	
	\begin{lemma}\label{selfreflectionL}
		Let $p\in\mathbb{P}$ and $X\in\mathcal{M}_p\cap\mathcal{L}$. Then, $p$ is an $(\mathcal{L},X)$-reflection of itself.
	\end{lemma}
	

	\begin{lemma}\label{amalgamation1L}
		Let $p\in\mathbb{P}$ and let $X\in\mathcal{M}_p\cap\mathcal{L}$. Let $q\in\mathbb{P}\cap X$ be an $(\mathcal{L},X)$-reflection of $p$. Then there is a condition $r\in\mathbb{P}$ extending both $p$ and $q$.
	\end{lemma}
	\begin{proof}
		Let $\mathcal{M}_r$ be the $(\mathcal{S}, \mathcal{L})$-symmetric system extending both $\mathcal{M}_p$ and $\mathcal{M}_q$ given by Lemma \ref{lemma-amal}. Let  
		\[
		\mathcal{A}_r:=\mathcal{A}_p\cup\mathcal{A}_q\cup \{Y\cap M\,:\, M\in\mathcal{A}_p\cup\mathcal{A}_q,\,Y\in\mathcal{M}_r\cap\mathcal{L}\cap M\},
		\]
		and let $a_r:=a_p\cup a_q$ and $d_r:=d_p=d_q$. Similarly as in the proof of Lemma \ref{amalgamation1S}, let also $b_r$ be the extension of $b_p\cup b_q$ to $\{(\alpha, \beta)\in a_r\times a_r, \alpha<\beta\}$ defined by letting $b_r(\alpha, \beta)$ be, for all $\alpha<\beta$ in $a_r$ with $\alpha\in a_p$ iff $\beta\notin a_p$, equal to $d_r\cap\delta_{M'}$ for any $M'\in\mathcal{A}_r$ such that $\alpha$, $\beta\in M'$ of minimal $\omega_1$-height (and letting $b_r(\alpha, \beta)=d_r$ if there is no such $M'$).
		
		For every $\nu\in d_r$ and every $\alpha\in a_r$, we define $u_r^{\alpha_i}(\nu)$ as follows:
		\[
		u_r^{\alpha_i}(\nu)=
		\begin{cases}
			u_q^{\alpha_i}(\nu), &\text{if }\alpha_i\in a_q\\
			u_p^{\alpha_i}(\nu), &\text{if }\alpha_i\in a_p.
		\end{cases}
		\]
		Let $u_r$ be the sequence $(u_r^\alpha:\alpha\in a_r)$. It is not too hard to see that $r:=(\mathcal{M}_r,\mathcal{A}_r,a_r,d_r,u_r, b_r)$ extends $p$ and $q$, and satisfies clauses (C1)-(C6) from the definition of the forcing $\mathbb{P}$, hence we only need to check that it also satisfies clauses (C7) and (C8).
		
		We start with the proof of (C7). Fix some $\nu\in d_r$ and let $\alpha,\beta\in a_r$ be such that $\alpha<_{r,\nu}\beta$. We will show that $u_r^\alpha(\nu)\leq u_r^\beta(\nu)$. Let $\overline{M}_0,\dots,\overline{M}_m\in\mathcal{A}_r^\nu$ and $\gamma_0<\dots<\gamma_{m-1}<\omega_3$ in $(\alpha,\beta)$ such that $\alpha\in \overline{M}_0$, $\beta\in\overline{M}_m$, and $\gamma_i\in\overline{M}_i\cap\overline{M}_{i+1}$ for each $i<m$. Note that if $\overline{M}_i$ is of the form $Y\cap M$, where $M\in\mathcal{A}_p^\nu\cup\mathcal{A}_q^\nu$ and $Y\in\mathcal{M}_r\cap\mathcal{L}\cap M$, we can substitute $\overline{M}_i$ by $M$, and $\alpha$ and $\beta$ remain $\mathcal{A}_r^\nu$-connected, as witnessed by the same sequence $\gamma_0,\dots,\gamma_{m-1}$ of ordinals of $\omega_3$. Hence, we may assume that $\overline{M}_0,\dots,\overline{M}_m\in\mathcal{A}_p^\nu\cup\mathcal{A}_q^\nu$. 
		
		Let $\{M_0,\dots,M_m\}\subseteq\mathcal{A}_p^\nu$ and $\{M_0',\dots,M_m'\}\subseteq\mathcal{A}_q^\nu$ be two enumerations such that for every $i\leq m$, if $\overline{M}_i\in\mathcal{A}_p^\nu$, then $\overline{M}_i=M_i$, and if $\overline{M}_i\in\mathcal{A}_q^\nu$, then $\overline{M}_i=M_i'$, and moreover assume that the enumeration of the models reflects the agreement given by items (3) and (4) of the definition of $(\mathcal{L},X)$-reflection. That is, assume that for every $i\leq m$,  
		\begin{enumerate}[label=(\roman*)]
			\item $M_i=M_i'$, if $M_i\in\mathcal{A}_p^\nu\cap X$,
			
			\item $X\cap M_i\cap\omega_3=X\cap M_i'\cap\omega_3=M_i'\cap\omega_3$,
			
			\item $a_q\cap M_i=a_p\cap M_i\cap X$, and
			
			\item $\delta_{M_i'}=\delta_{M_i}$.
			
		\end{enumerate}
		
		We first show that $\alpha,\beta\in a_p$. Let $\eta\in\{\alpha,\beta\}$ and $N\in\{\overline{M}_0,\overline{M}_m\}$ be such that $\eta\in M$, and assume that $\eta\in a_q$. If $N\in\mathcal{A}_p^\nu$, then $\eta\in a_q\cap N=a_p\cap N\cap X$, by (iii). If $N\in\mathcal{A}_q^\nu$, since $\eta\in a_q\subseteq X$, by (ii) there is some $N^*\in\mathcal{A}_p^\nu$ such that $\eta\in N\cap\omega_3=X\cap N^*\cap\omega_3$. Hence, by (iii), we have $\eta\in a_q\cap N^*=a_p\cap N^*\cap X$. Since $a_r=a_p\cup a_q$, this finishes the proof of $\alpha,\beta\in a_p$. 
		
		Now, note that for every $i\leq m$ such that $\overline{M}_i=M_i'\in\mathcal{A}_q^\nu$, in light of item (ii) above, 
		\[
		\gamma_i,\gamma_{i+1}\in M_i'\cap\omega_3=X\cap M_i\cap\omega_3\subseteq M_i\cap\omega_3.
		\]
		Therefore, by (iv), $\alpha$ and $\beta$ are $\mathcal{A}_p^\nu$-connected through $\{M_0,\dots,M_m\}$ (in symbols $\alpha<_{p,\nu}\beta$), and hence, since $\alpha,\beta\in a_p$, we have 
		\[
		u_r^\alpha(\nu)=u_p^\alpha(\nu)\leq u_p^\beta(\nu)=u_r^\beta(\nu).
		\]
		
		Finally, clause (C8) from the definition of $\mathbb P$-condition holds for $r$, exactly as in the proof of Lemma \ref{amalgamation1S}, thanks to (C7) holding and by the definition of $b_r(\alpha, \beta)$ in the case that $\alpha\in a_p$ iff $\beta\notin a_p$. 
	\end{proof}

	\begin{lemma}\label{propL}
		The forcing $\mathbb{P}$ is proper for countably closed models of size $\aleph_1$.
	\end{lemma}
	\begin{proof}
		Let $p\in\mathbb{P}$ and let $X^*$ be a countably closed $\aleph_1$-sized elementary submodel of a large enough $H(\theta)$ containing $p$ and $\mathbb{P}$. Let $X:=X^*\cap H(\omega_3)$. By Lemma \ref{ontop}, we may find a condition $p'$ extending $p$ and such that $X\in\mathcal{M}_{p'}$. We will show that $p'$ is $(X^*,\mathbb{P})$-generic. So, let $D\in X^*$ be a dense subset of $\mathbb{P}$, and let $p^*\in D$ be such that $p^*\leq p'$. As in the proof of Lemma \ref{propS}, since $p^*$ is an $(\mathcal{L},X)$-reflection of itself by Lemma \ref{selfreflectionL}, if we show that all the parameters in Definition \ref{defreflectionL} are in $X^*$, and then argue by elementarity, we can find an $(\mathcal{L},X)$-reflection $q\in D\cap X^*$ of $p^*$. Since $p^*$ and $q$ are compatible by Lemma \ref{amalgamation1L}, this is enough to show that $p^*$ is $(X^*,\mathbb{P})$-generic. Seeing that all the parameters in the definition of $(\mathcal{L},X)$-reflection are in $X^*$ is much easier in this case than in the proof of Lemma \ref{propS}. Indeed, note that the parameters in items (3.b) and (3.c) of Definition \ref{defreflectionL} are in $X$ because it is a countably closed model, and note that the other items are obvious or follow from the fact that $\omega_1\subseteq X$. 
	\end{proof}

	\subsection{The chain condition}\label{subsect-sc-chain-condition}
	
	In our final lemma we will show that $\mathbb{P}$ has the $\aleph_3$-chain condition. We will in fact prove that $\mathbb{P}$ has the $\aleph_3$-Knaster condition; in other words, that for every sequence $(p_\xi\,:\,\xi<\omega_3)$ of conditions in $\mathbb{P}$ there is $X\subseteq\omega_3$ of size $\aleph_3$ such that $p_{\xi_0}$ and $p_{\xi_1}$ are compatible for al $\xi_0,\xi_1<\omega_3$ in $\mathbb{P}$.

	\begin{lemma}\label{cc}
		$\mathbb{P}$ has the $\aleph_3$-Knaster condition.
	\end{lemma}
	
	\begin{proof}
		Fix a condition $p_\xi\in\mathbb{P}$ for every $\xi<\omega_3$. For each $\xi<\omega_3$, let $X_\xi$ be a large elementary submodel of $(H(\omega_3); \in, \vec e)$ such that $p_\xi\in X_\xi$. By $2^{\aleph_1}=\aleph_2$ (recall that $\mathrm{GCH}$ was assumed at the beginning of the section), we may find $I\in [\omega_3]^{\aleph_3}$ such that for all $\xi_0<\xi_1$ in $I$, the structures $(X_{\xi_0}; \in, p_{\xi_0})$ and $(X_{\xi_1}; \in, p_{\xi_1})$ are isomorphic, and the isomorphism $\Psi_{X_{\xi_0}, X_{\xi_1}}$ is the identity on $X_{\xi_0}\cap X_{\xi_1}$. We will now prove that for all $\xi_0<\xi_1$ in $I$, the conditions $p_{\xi_0}$ and $p_{\xi_1}$ are compatible. 
		
		By Lemma \ref{pureamalgamation2}, $\mathcal{M}_{p_{\xi_0}}\cup\mathcal{M}_{p_{\xi_1}}$ is an $(\mathcal{S},\mathcal{L})$-symmetric system. Moreover, note that if $M\in\mathcal{A}_{p_{\xi_1}}\cap\mathcal{S}$ and $Y\in\mathcal{M}_{p_{\xi_0}}\cap\mathcal{L}\cap M$, then $Y\in X_{\xi_1}$, and hence, $\Psi_{X_{\xi_0},X_{\xi_1}}(Y)=Y$. Therefore, $Y\in\mathcal{M}_{p_{\xi_1}}$, and by clause (C2) applied to $p_{\xi_1}$, we have $Y\cap M\in\mathcal{A}_{p_{\xi_1}}$. Hence, we can conclude that $\mathcal{A}_{p_{\xi_0}}\cup\mathcal{A}_{p_{\xi_1}}$ satisfies clause (C2). Denote $\mathcal{A}_{p_{\xi_0}}\cup\mathcal{A}_{p_{\xi_1}}$ by $\mathcal{A}_q$ from now on. 
		
		Since $X_{\xi_1}$ is a large model, $\omega_1\subseteq X_{\xi_1}$. Hence, in particular, $d_{p_{\xi_0}}\subseteq X_{\xi_1}$, and thus, $\Psi_{X_{\xi_0},X_{\xi_1}}"d_{p_{\xi_0}}=d_{p_{\xi_0}}$, which implies that $d_{p_{\xi_0}}=d_{p_{\xi_1}}$. Define $a_q:=a_{p_{\xi_0}}\cup a_{p_{\xi_1}}$ and $d_q:=d_{p_{\xi_0}}=d_{p_{\xi_1}}$, and note that for every $\nu\in d_q$, if $\alpha\in a_{p_{\xi_0}}\cap a_{p_{\xi_1}}$, then $u_{p_{\xi_0}}^\alpha(\nu)=u_{p_{\xi_1}}^\alpha(\nu)$. For all $\nu\in d_q$ and all $\alpha\in a_{p_{\xi_i}}$, where $i\in\{0,1\}$, define $u_q^\alpha(\nu)=u_{p_{\xi_i}}^\alpha(\nu)$, and let $u_q$ be the sequence $(u_q^\alpha:\alpha\in a_q)$. It is clear that $u_q^\alpha$ is a well-defined function extending $u_{p_{\xi_i}}^\alpha$, for $i\in\{0,1\}$ and all $\alpha\in a_{p_{\xi_i}}$. Finally, let $b_q=(b_q(\alpha, \beta)\,:\,\alpha, \beta\in a_q, \alpha<\beta)$ be the function extending $b_{p_{\xi_0}}$ and $b_{p_{\xi_1}}$ by letting $b_q(\alpha, \beta)=d_q$ whenever $\alpha\in u_{p_{\xi_0}}$ iff $\beta\notin u_{p_{\xi_1}}$. If we show that
		\[
		q:=(\mathcal{M}_{p_{\xi_0}}\cup\mathcal{M}_{p_{\xi_1}},\mathcal{A}_q, a_q,d_q,u_q, b_q)
		\]
		satisfies clause (C7), then $q$ will be a common extension of $p_{\xi_0}$ and $p_{\xi_1}$ in $\mathbb{P}$, as we wanted.
		
		Let $\nu\in d_q$, and $\alpha<\beta$ in $a_q$ such that $\alpha <_{\mathcal{A}_q,\nu}\beta$. Let $M_0,\dots, M_n\in\mathcal{A}_q^\nu$ and $\gamma_0<\dots<\gamma_{n-1}<\omega_3$ in $(\alpha,\beta)$ such that $\alpha\in M_0$, $\beta\in M_n$, and $\gamma_i\in M_i\cap M_{i+1}$ for each $i<n$. We will show that $u_q^\alpha(\nu)\leq u_q^\beta(\nu)$. Now, for every $i\leq n$, if $M_i\in\mathcal{A}_{p_{\xi_0}}$, we let $M_i'$ denote the model $M_i$, and if $M_i\in\mathcal{A}_{p_{\xi_1}}$, we let $M_i'$ denote the model $\Psi_{X_{\xi_1},X_{\xi_0}}(M_i)$. Moreover, for every $i<n$, if either $M_i$ or $M_{i+1}$ belong to $\mathcal{A}_{p_{\xi_1}}$, we let $\gamma_i':=\Psi_{X_{\xi_1},X_{\xi_0}}(\gamma_i)$. Otherwise, we let $\gamma_i':=\gamma_i$. Finally, if $\alpha\in a_{p_{\xi_0}}$, let $\alpha':=\alpha$, and if $\alpha\in a_{p_{\xi_1}}$, let $\alpha':=\Psi_{X_{\xi_1},X_{\xi_0}}(\alpha)$, and similarly for $\beta$. 
		
		We claim that $\sup_{i\leq n}\delta_{M_i'}\leq\nu$, $\alpha'<\gamma_0'<\dots<\gamma_{n-1}'<\beta'$, $\alpha'\in M_0'$, $\beta'\in M_n'$, and $\gamma_i'\in M_i'\cap M_{i+1}'$ for each $i<n$. In particular, we claim that $\alpha'<_{\mathcal{A}_{p_{\xi_0}},\nu}\beta'$, which in turn will imply $u_q^\alpha(\nu)=u_{p_{\xi_0}}^{\alpha'}(\nu)\leq u_{p_{\xi_0}}^{\beta'}(\nu)=u_q^\beta(\nu)$, because $\Psi_{X_{\xi_1},X_{\xi_0}}$ is an isomorphism.
		
		It is clear that $\sup_{i\leq n}\delta_{M_i'}\leq\nu$ and that $\gamma_i'\in M_i'\cap M_{i+1}'$, for each $i<n$. Moreover, note that if $\alpha\in a_{p_{\xi_0}}$ and $M_0\in\mathcal{A}_{p_{\xi_0}}$, then $\alpha'=\alpha\in M_0=M_0'$, and if $\alpha\in a_{p_{\xi_1}}$ and $M_1\in\mathcal{A}_{p_{\xi_1}}$, then $\alpha'=\Psi_{X_{\xi_1},X_{\xi_0}}(\alpha)\in\Psi_{X_{\xi_1},X_{\xi_0}}(M_i)=M_i'$. If $\alpha\in a_{p_{\xi_0}}$ and $M_0\in\mathcal{A}_{p_{\xi_1}}$, then $\alpha\in X_{\xi_0}\cap X_{\xi_1}$, and hence, $\alpha'=\alpha=\Psi_{X_{\xi_1},X_{\xi_0}}(\alpha)\in\Psi_{X_{\xi_1},X_{\xi_0}}(M_0)=M_0'$. If $\alpha\in a_{p_{\xi_1}}$ and $M_0\in\mathcal{A}_{p_{\xi_0}}$, then $\alpha\in X_{\xi_0}\cap X_{\xi_1}$, and thus, $\alpha'=\Psi_{X_{\xi_1},X_{\xi_0}}(\alpha)=\alpha\in M_0=M_0'$. The same argument shows that $\beta'\in M_n'$. 
		
		It only remains to check that $\alpha'<\gamma_0'<\dots<\gamma_{n-1}'<\beta'$. First, we show that $\alpha'<\gamma_0'$. If $M_0\in\mathcal{A}_{p_{\xi_0}}$, then $\alpha'=\alpha$. Moreover, if $M_1\in\mathcal{A}_{p_{\xi_0}}$, then $\gamma_0'=\gamma_0$, and if $M_1\in\mathcal{A}_{p_{\xi_1}}$, then $\gamma_0\in X_{\xi_0}\cap X_{\xi_1}$, and thus, $\gamma_0'=\Psi_{X_{\xi_1},X_{\xi_0}}(\gamma_0)=\gamma_0$. In both cases $\alpha'=\alpha$ and $\gamma_0'=\gamma_0$, so we have that $\alpha'<\gamma_0'$. Now, if $M_0\in\mathcal{A}_{p_{\xi_1}}$, then $\alpha'=\Psi_{N_{\xi_1},N_{\xi_0}}(\alpha)$ and $\gamma_0'=\Psi_{N_{\xi_1},N_{\xi_0}}(\gamma_0)$, and hence $\alpha'<\gamma_0'$. Next, we check that $\gamma_i'<\gamma_{i+1}'$ for every $i<n$. By similar reasons as above, it is not too hard to see that if $M_{i+1}\in\mathcal{A}_{p_{\xi_0}}$, then  $\gamma_i'=\gamma_i$ and $\gamma_{i+1}'=\gamma_{i+1}$, and if $M_{i+1}\in\mathcal{A}_{p_{\xi_1}}$, then $\gamma_i'=\Psi_{X_{\xi_1},X_{\xi_0}}(\gamma_i)$ and $\gamma_{i+1}'=\Psi_{X_{\xi_1},X_{\xi_0}}(\gamma_{i+1})$. In both cases $\gamma_i'<\gamma_{i+1}'$. Lastly, we check that $\gamma_{n-1}'<\beta'$. But again, by the same reasons as before, if $M_n\in\mathcal{A}_{p_{\xi_0}}$, then $\gamma_{n-1}'=\gamma_{n-1}$ and $\beta'=\beta$, and if $M_n\in\mathcal{A}_{p_{\xi_1}}$, then $\gamma_{n-1}'=\Psi_{X_{\xi_1},X_{\xi_0}}(\gamma_{n-1})$ and $\beta'=\Psi_{X_{\xi_1},X_{\xi_0}}(\beta)$. So, in both cases $\gamma_{n-1}'<\beta'$. 
		
		This finishes the proof of $\alpha'<_{\mathcal{A}_{p_{\xi_0}},\nu}\beta'$ and, as we observed above, implies that $u_q^\alpha(\nu)=u_{p_{\xi_0}}^{\alpha'}(\nu)\leq u_{p_{\xi_0}}^{\beta'}(\nu)=u_q^\beta(\nu)$. Therefore, this shows that $q$ satisfies clause (C7), and hence that $q$ is a condition in $\mathbb{P}$ extending both $p_{\xi_0}$ and $p_{\xi_1}$.
	\end{proof}

	\subsection{Conclusion}
	
	Recall that we have assumed $\mathrm{GCH}$ throughout the last section. Therefore, under this assumption, $\mathcal{L}$ is stationary in $[H(\omega_3)]^{\aleph_1}$, and hence Lemmas \ref{preservacio-proper}, \ref{propS}, \ref{propL} and \ref{cc} ensure that the forcing $\mathbb{P}$ preserves all cardinals. 
	
	We will finish by showing that $\mathbb{P}$ in fact forces a strong $\omega_3$-chain of subsets of $\omega_1$, which finishes the proof of Theorem \ref{mainthm}.
	
	Let $G$ be a $\mathbb{P}$-generic filter over $V$ and let us work in $V[G]$. For all $\alpha<\omega_3^V=\omega_3^{V[G]}$, if there is any $p\in G$ such that $\alpha\in a_p$, we let
	\[
	C_\alpha:=\{p\in\mathbb{P}:p\in G,\alpha\in a_p\},
	\]
	and
	\[
	u_G^\alpha:=\bigcup\{u_p^\alpha\,:\,  p\in C_\alpha\}.
	\]
	By Lemma \ref{dens0}, the set $B$ of ordinals $\alpha<\omega_3$ for which $u_G^\alpha$ is defined has size $\aleph_3$. Note that for all $\alpha<\beta$ in $B$, $\dom(u_G^\alpha)=\dom(u_G^\beta)$ by genericity. Let now $A:=\dom(u_G^\alpha)$ for any $\alpha\in B$,  and observe that $A\in[\omega_1]^{\omega_1}$ by Lemma \ref{dens1}. Let $(\alpha_i\,:\, i<\omega_3)$ be the strictly increasing enumeration of $B$ and let $(\nu_\tau\,:\,\tau<\omega_1)$ be the strictly increasing enumeration of $A$. For each $i<\omega_3$ let 
	\[
	X_i=\{\tau\in \omega_1\,:\, u_G^{\alpha_i}(\nu_\tau)=1\}
	\]
	Using our density lemmas \ref{dens1} and \ref{dens2} it is now a routine matter to check that $(X_i\,:\, i<\omega_3)$ is a strong $\omega_3$-chain of subsets of $\omega_1$ -- the fact that $X_{i_0}\setminus X_{i_1}$ is finite for all $i_0<i_1$ follows of course from the presence of the component $b_p$ in the definition of $\mathbb P$-condition. This finishes the proof of Theorem \ref{mainthm}.

	\bibliographystyle{plainurl} 
	\bibliography{References}

\end{document}